\documentclass[reqno]{amsart}

\usepackage{amsthm}   
\usepackage{amsfonts, amsmath, amscd}
\usepackage{amssymb, latexsym}
\usepackage[all,cmtip]{xy}     
\usepackage{enumerate} 
\usepackage{color}
\usepackage[usenames,dvipsnames]{xcolor}
\usepackage{verbatim}
\usepackage{xcolor}
 
\usepackage{hyperref}
 \hypersetup{colorlinks=false}
 \usepackage{cite}
 \usepackage{ amssymb }

 \usepackage{graphicx, psfrag}
 \usepackage{pstool}

    \usepackage{xspace}   
    \usepackage{tikz-cd}
\usepackage{tikz}
\usetikzlibrary{arrows}

\usepackage[titletoc]{appendix}
\usepackage[normalem]{ulem}     
\usepackage{amsthm}
\usepackage{amsmath}
\usepackage{amscd}
\usepackage[latin2]{inputenc}
\usepackage{t1enc}
\usepackage[mathscr]{eucal}
\usepackage{indentfirst}
\usepackage{graphicx}
\usepackage{graphics}
\usepackage{pict2e}
\usepackage{epic}

\usepackage{amssymb}
\usepackage{amsmath}
\usepackage{amsfonts}
\usepackage{amscd}
\usepackage{amsthm}
\usepackage{graphicx,psfrag}
\usepackage{wrapfig}
\usepackage{subeqnarray}
\usepackage{amssymb}
\usepackage{yfonts}
\usepackage[makeroom]{cancel}
\usepackage{trimclip}
\usepackage{cancel}
\newif\iflclip
\newif\ifbclip
\newif\ifrclip
\newif\iftclip
\def\CLIP{\dimexpr\fboxrule+.2pt\relax}
\def\nulclip{0pt}
\newcommand\partbox[2]{%
  \lclipfalse\bclipfalse\rclipfalse\tclipfalse%
  \let\lkern\relax\let\rkern\relax%
  \let\lclip\nulclip\let\bclip\nulclip\let\rclip\nulclip\let\tclip\nulclip%
  \parseclip#1\relax\relax%
  \iflclip\def\lkern{\kern\CLIP}\def\lclip{\CLIP}\fi
  \ifbclip\def\bclip{\CLIP}\fi
  \ifrclip\def\rkern{\kern\CLIP}\def\rclip{\CLIP}\fi
  \iftclip\def\tclip{\CLIP}\fi
  \lkern\clipbox{\lclip{} \bclip{} \rclip{} \tclip}{\fbox{#2}}\rkern%
}
\def\parseclip#1#2\relax{%
  \ifx l#1\lcliptrue\else
  \ifx b#1\bcliptrue\else
  \ifx r#1\rcliptrue\else
  \ifx t#1\tcliptrue\else
  \fi\fi\fi\fi
  \ifx\relax#2\relax\else\parseclip#2\relax\fi
}
\usepackage{slashed} 

\theoremstyle{definition}

\theoremstyle{remark}

\numberwithin{equation}{section}

\usepackage{amsmath}
\usepackage{amssymb}
\usepackage{amsfonts}
\usepackage{amsthm}
\usepackage{mathtools}
\usepackage{graphicx}
\usepackage{colonequals}
\usepackage{booktabs}
\usepackage{cancel}
\usepackage[numbers]{natbib}
\usepackage{hyperref}
\usepackage{fancyhdr}
\usepackage{multirow}
\usepackage{multirow}
\usepackage{braket}
\usepackage{tikz}
\usetikzlibrary{arrows,shapes,decorations.pathmorphing}
\usepackage{tikz-cd}
\usepackage[margin=2.9cm]{geometry}
\usepackage{mathabx}
\usepackage{ esint }
\usepackage{slashed}
\usepackage{mathrsfs}
\usepackage{amssymb}
\usepackage{blkarray}
\usepackage{xcolor}
\usepackage{pinlabel}   
\usepackage{scalerel,stackengine}
\stackMath
\newcommand\reallywidehat[1]{%
\savestack{\tmpbox}{\stretchto{%
  \scaleto{%
    \scalerel*[\widthof{\ensuremath{#1}}]{\kern-.6pt\bigwedge\kern-.6pt}%
    {\rule[-\textheight/2]{1ex}{\textheight}}
  }{\textheight}%
}{0.5ex}}%
\stackon[1pt]{#1}{\tmpbox}%
}
\newcommand{\R}{\mathbb{R}}

\newcommand{\Z}{\mathbb{Z}}

\newcommand{\C}{\mathbb{C}}

\newcommand{\br}{\langle}
\newcommand{\kt}{\rangle}

\usepackage{graphicx}

\theoremstyle{definition}
\newtheorem{thm}{Theorem}[section]
\newtheorem{prop}[thm]{Proposition}
\newtheorem{lm}[thm]{Lemma}
\newtheorem{defn}[thm]{Definition}
\newtheorem{rem}[thm]{Remark}
\newtheorem{eg}[thm]{Example}

\newtheorem{cor}[thm]{Corollary}

\newtheorem*{hyp1*}{Hypothesis (I)}
\newtheorem*{hyp2*}{Hypothesis (II)}
\newtheorem*{hyp3*}{Hypothesis (III)}

\newtheoremstyle{exercise}{}{}{\itshape}{}{\bfseries}{:}{.5em}{\thmname{#1} \thmnumber{#2}\thmnote{(#3)}}
\theoremstyle{exercise}

\usepackage{amsfonts}
\DeclareFontFamily{U}{wncy}{}
\DeclareFontShape{U}{wncy}{m}{n}{<->wncyr10}{}
\DeclareSymbolFont{mcy}{U}{wncy}{m}{n}
\DeclareMathSymbol{\sha}{\mathord}{mcy}{"58}

\renewcommand{\d}[2]
{\frac{d#1}{d#2}}

\newcommand{\nc}{\newcommand}
\nc {\Prod}{(\underset{I}\ph\, r_I^{n_I})}
\nc {\wlw}{\wedge\ldots\wedge}
\nc {\olo}{\otimes\ldots\otimes}
\nc{\bea}{\begin{eqnarray*}}
\nc{\eea}{\end{eqnarray*}}
\nc {\e}{\varepsilon}
\nc{\de}{\delta}
\nc{\A}{\hat a}
\nc{\Ad}{\hat a^\dag}
\nc{\grad}{\nabla}
\nc{\nlim}{\underset{n\to\infty}\lim}
\nc{\ilim}{\underset{i\to\infty}\lim}
\nc{\isum}{\sum\limits_{i=1}^N}
\nc{\xx}{\underset{x\to x_0}\lim}
\nc{\Bnrm}{\Big |\Big|}
\nc{\sym}{\text{Sym}}
\nc{\inte}{\overset{\circ}}
\nc{\del}{\partial}
\nc{\plp}{+\ldots+}
\nc{\lre}{\longrightarrow}
\nc{\be}{\begin{equation}}
\nc{\ee}{\end{equation}}
\nc{\CP}{\mathbb{CP}}
\nc{\End}{\text{End}}
\nc{\mf}{\mathfrak}
\nc{\Hom}{\text{Hom}}
\nc{\spec}{\text{Spec}}
\nc{\sub}{\subseteq}
\nc{\weakto}{\rightharpoonup}
\nc{\ph}{\varphi}
\nc{\leqc}{\lesssim}
\nc{\delbar}{\overline \del}
\nc{\Tr}{\text{Tr}}
\nc{\Lhe}{\mathcal L_{(\Phi^{h_\e}, A^{h_\e})}}

\begin{document}


\title{Concentrating Dirac Operators and Generalized Seiberg-Witten Equations}

\author{Gregory J. Parker}
\address{Department of Mathematics, Stanford University}
\email{gjparker@stanford.edu}

\begin{abstract}
This article studies a class of Dirac operators of the form $D_\e= D+\e^{-1}\mathcal A$, where $\mathcal A$ is a zeroth order perturbation vanishing on a subbundle. When $\mathcal A$ satisfies certain additional assumptions, solutions of the Dirac equation have a concentration property in the limit $\e\to 0$: components of the solution orthogonal to $\ker(\mathcal A)$ decay exponentially away from the locus $\mathcal Z$ where the rank of $\ker(\mathcal A)$ jumps up. These results are extended to a class of non-linear Dirac equations. 

This framework is then applied to study the compactness properties of moduli spaces of solutions to generalized Seiberg-Witten equations. In particular, it is shown that for sequences of solutions which converge weakly to a $\Z_2$-harmonic spinor, certain components of the solutions concentrate exponentially around the singular set of the $\Z_2$-harmonic spinor. Using these results, the weak convergence to $\Z_2$-harmonic spinors proved in existing convergence theorems (\cite{Taubes3dSL2C, Taubes4dSL2C,HWCompactness,TaubesU1SW,WalpuskiZhangCompactness}) is improved to $C^\infty_{loc}$.

\end{abstract}


\maketitle
\tableofcontents

\section{Introduction}

Let $(Y,g)$ be a Riemannian manifold of dimension $n\geq 3$, and $ D: \Gamma(E)\to \Gamma(E)$ a Dirac operator on sections of a Clifford module $E\to Y$, i.e. a first-order elliptic operator whose principal symbol satisfies $\sigma_D^2=-\text{Id}$. A 1-parameter family of Dirac operators displaying a concentration property or, more succinctly, a {\bf Concentrating Dirac Operator} (sometimes called a {\it localizing} Dirac operator) is a parameterized perturbation \be D_\e = D + \tfrac{1}{\e}\mathcal A\label{concentratingDiracprelim}\ee
\noindent of $D$ by positive scalings of a zeroth order operator $\mathcal A \in \text{End}(E)$ such that the support of solutions concentrates along a distinguished collection of submanifolds $\mathcal Z\subset Y$ as $\e\to 0$. Concentrating Dirac operators were introduced to the mathematical literature by Witten's celebrated work on Morse theory \cite{WittenMorseTheory}, although they were familiar to physicists for decades prior to this. Since then, concentrating Dirac operators have been employed to give geometric proofs of many results in index theory \cite{ManosConcentrationI,ManosConcentrationII,RussianConcentratingIndex,JinLocalizingIndex,LocalizingIndexI, LocalizingIndexII,ZhangChernWeilWitten}, and  geometric quantization  \cite{FurutaTorusI,FuturaTorusII, FurutaTorusIII,TianAnalytic}. 

Concentrating Dirac operators have also played a significant role in Seiberg-Witten gauge theory. It is an observation due to C. Taubes that the linearization of the Seiberg-Witten equations behaves as a concentrating Dirac operator in certain limits. This perspective is central to Taubes's celebrated work ``SW=Gr'' relating the Seiberg-Witten and Gromov invariants of symplectic 4-manifolds \cite{TaubesSW=Gr,SW=Gr1, SW=Gr2,SW=Gr3,SW=Gr4}, and to his resolution of the Weinstein Conjecture \cite{TaubesWeinstein}. In these situations, the role of the perturbation $\mathcal A$ is occupied by a large multiple of the symplectic or contact form, and the solutions concentrate along submanifolds of codimension 2 which Taubes proves are the pseudo-holomorphic curves enumerated by the Gromov invariant or the Reeb orbits whose existence was postulated by Weinstein in the two cases respectively. 

Concentrating Dirac operators are also relevant in more recent work of Taubes and others on the compactness of moduli spaces for generalized Seiberg-Witten theories. For a specified compact Lie group $G$, a system of {\bf generalized Seiberg-Witten equations} on a 3 or 4-dimensional manifold is a system of first-order non-linear PDEs for a pair $(\Psi, A)$ of a spinor $\Psi$ and a connection $A$ on a principal $G$-bundle which has the schematic form 
\begin{eqnarray}
\slashed D_A \Psi &=& 0\label{prelimSW1} \\
\star F_A  &=&- \tfrac{1}{2}\mu(\Psi,\Psi) \label{prelimSW2}
\end{eqnarray}
\noindent where $\slashed D_A$ (resp. $\slashed D_A^+$ on a 4-manifold) is the Dirac operator twisted by the connection $A$, $F_A$ (resp. $F_A^+$) is the curvature of $A$ (resp. the self-dual component thereof), and $\mu$ is a pointwise quadratic map \footnote{Note that my convention for the sign of $\mu$ differs from that used by many authors. That is to say, I denote by $-\mu$ what others denote by $\mu$; the equations (thus their relevant compactness properties) are the same.}.  Examples include the Vafa-Witten equations  \cite{TaubesVW,TT1,TT2}, the Kapustin-Witten equations \cite{TaubesKWNahmPole, MazzeoWittenNahmI, MazzeoWittenNahmII,WittenFivebranesKnots}, the $\text{SL}(2,\C)$ anti-Self-Dual Yang-Mills equations \cite{Taubes3dSL2C, Taubes4dSL2C}, the Seiberg-Witten equations with multiple spinors \cite{HWCompactness,TaubesU1SW}, and the ADHM Seiberg-Witten Equations \cite{WalpuskiZhangCompactness}. The reader is referred to  \cite{WittenKhovanovGaugeTheory, DonaldsonSegal, DWAssociatives,HaydysG2SW,WalpuskiNotes} for discussions of conjectures relating these equations to the geometry of manifolds and to other gauge theories. The main barrier to progress on all of these conjectures is the lack of a well-understood compactification for the moduli space of solutions: 
unlike for the standard Seiberg-Witten equations, these more general theories do not admit compact moduli space of solutions and there instead may be sequences of solutions for which the $L^2$-norm of the spinor diverges. For such sequences, a renormalized (sub)sequence must converge to a {\bf $\Z_2$-harmonic spinor} or more general {\bf Fueter section} (proved for each respective equation in the above references).

A promising approach to constructing well-understood compactifications for these moduli spaces is to attach boundary strata consisting of $\Z_2$-harmonic spinors or Fueter sections. A necessary step in showing the suitability of any putative compactification formed in this way is to construct boundary charts for the moduli space. Constructing these charts requires gluing results (see \cite{DWDeformations,PartI,PartII,PartIII} for progress in this direction). Even after showing appropriate gluing results, however, the existence of the desired charts does not follow immediately since it is not {\it a priori} clear that any sequence approaching a given boundary point necessary arises from such a gluing---in other words, the gluing may only construct a subset of the desired chart. In order for the compactifications to be well-behaved, extraneous sequences not captured by gluing constructions must be ruled out: this problem is known as the {\it surjectivity of gluing}. The convergence to $\Z_2$-harmonic spinors and Fueter sections proved in extant convergence results is rather weak (see Section \ref{section4.2} for a precise statement); in particular it leaves open the possibility for the existence of sequences converging to a $\Z_2$-harmonic spinor or Fueter section in a space of low regularity that would necessarily elude gluing constructions (which automatically have convergence in higher-regularity spaces).

  As explained in Section \ref{section5}, attempts to bootstrap the convergence to $\Z_2$-harmonic spinors using standard methods are doomed to fail by the accumulation of powers of the spinor's diverging $L^2$-norm in the relevant estimates, and more robust techniques are required. The theory of concentrating Dirac operators supplies these techniques: for a sequence of solutions to (\refeq{prelimSW1}--\refeq{prelimSW2}) converging to a $\Z_2$-harmonic spinor, the linearization of the equations behaves as a particular type of concentrating Dirac operator, with the diverging $L^2$-norm of the spinor occupying the role of $\e^{-1}$ in the expression (\refeq{concentratingDiracprelim}). Although the perspective and philosophy of concentrating Dirac operators implicitly informs the approach of  \cite{Taubes3dSL2C, Taubes4dSL2C,HWCompactness,TaubesU1SW,WalpuskiZhangCompactness, TaubesVW, TaubesKWNahmPole}, there is more to be gained by making the connection precise. 

This article is best viewed as consisting of three parts. First, Sections \ref{section2}--\ref{section3} extend results about the behavior of solutions to concentrating Dirac operators to a larger class of operators than previously studied, and to Dirac operators with certain types of non-linearities. Second, Sections \ref{section4}--\ref{newsection5}, show that the linearization of generalized Seiberg-Witten equations, by design, fall into this new class of operators. Finally, Section \ref{section5} uses these results to improve convergence results for sequences of solutions to generalized Seiberg-Witten equations in \cite{Taubes3dSL2C, Taubes4dSL2C,HWCompactness,TaubesU1SW,WalpuskiZhangCompactness}  to the $C^\infty_{loc}$ topology.  Although this bootstrapping is the main application given here, the exponential convergence results obtained from the results of Sections \ref{section2}--\ref{section3} are far stronger than is necessary simply for bootstrapping: they additionally provide a more precise geometric picture of how the convergence occurs.  In particular, as discussed in Appendix \ref{appendixA}, they imply that there is an expected invariant length scale for the concentration of curvature along the singular set. The results and techniques developed here may therefore be helpful in addressing questions related to gluing and the surjectivity of gluing (in fact, the present work grew out of the necessity for the full strength of these results in the gluing construction of \cite{PartI, PartIII}).

\bigskip 

 \subsection{Main Results}

The key property that leads to concentration as $\e\to 0$ is a commutation relation of $\mathcal A$ and the principal symbol $\sigma_D$ of the unperturbed operator $D$. These are required to satisfy $$\mathcal A^\star \sigma_D(\xi)=\sigma_D^\star(\xi) \mathcal A$$ for any $\xi \in T^\star Y$. In previous work on concentrating Dirac operators, it has typically been assumed that $\mathcal A$ is invertible on an open dense subset of $Y$; here this assumption is weakened to include the case that $\mathcal A$ vanishes identically along a subbundle of $E$. More precisely, let $r$ denote the maximal rank of $\mathcal A$, and set $\mathcal Z= \{y  \ | \ \text{rank}(\mathcal A(y))< r\}\subseteq Y$. Assume that there is a parallel decomposition $E|_{Y-\mathcal Z}= \frak N\oplus \frak H$ of the Clifford module's restriction to $Y-\mathcal Z$, in which  $\mathcal A$ takes the form  \be \mathcal A=\begin{pmatrix} 0 & 0 \\ 0 & A_\frak H\end{pmatrix}.\label{blockdiagonal}\ee
\noindent Dirac operators satisfying these two assumptions will be referred to as {\it concentrating Dirac operators with fixed degeneracy} (see Definition \ref{concentratingdiracdef} for a more precise definition). 

The first main result shows that the $\frak H=\ker(\mathcal A)^\perp$-components of a solution to the Dirac equation concentrate along $\mathcal Z$ and decay exponentially away from it: 

\begin{thm} \label{concentrationprinciple}\label{maina}
Suppose that $D_\e$ is a concentrating Dirac operator with fixed degeneracy, and  that $\frak q\in \Gamma(E)$ is a solution i.e. $$(D+ \tfrac{1}{\e}\mathcal A)\frak q=0.$$ 
For any compact subset $K\Subset Y-\mathcal Z$, let $R_K=\text{dist}(K,\mathcal Z)$. Then, there exists a compact subset $K'$ with $K\Subset K' \Subset Y-\mathcal Z$, and constants $C,c$ independent of $K$ such that for $\e$ sufficiently small, the components of $\frak q$ in the subbundle $\frak H$ obey 

\be
\|\pi_{\frak H}(\frak q)\|_{C^0(K)} \leq \frac{C}{R_K^{n/2}}  \ \text{Exp} \left(-\frac{c\Lambda_K}{\e}R_K\right) \|\frak q\|_{L^{1,2}(K')} \label{expdecay}
\ee   

\noindent where $\Lambda_K= \underset{y\in K, v\in \frak H_y} \inf\frac{ \|\mathcal A v\|}{\|v \|}$, i.e. $\Lambda_K$ is the minimum fiberwise norm of $A_\mathfrak H^{-1}$ over $K$. 
\end{thm}  

\medskip 

The proof of the above result occupies Section \ref{section2}. It is worth remarking that the proof is easily adapted to show that the conclusion of Theorem \ref{maina} also holds if $\frak q$ is an eigenspinor of $D_\e$ with eigenvalue $\lambda< \tfrac{c_K}{\e}$ for a small constant $c_K$ depending depending on the compact set $K$, though this extension is not needed here and the details of the proof are omitted. 

The arguments in the proof of Theorem \ref{maina} can be extended to the following class of non-linear concentrating Dirac operators.  Suppose that $Q: \Gamma(E)\to \Gamma(E)$ is a pointwise non-linear bundle map that takes the form $Q(\frak q)= Q_1(\frak q) \pi_\mathfrak H (\frak q)$ where $Q_1$ is again a pointwise non-linear bundle map and $\pi_\frak H$ the projection onto the subbundle $\frak H$. Said more simply,  it is assumed that the non-linearity has at least a linear factor in the $\mathfrak H$ components. Additionally, assume that $Q_1(\frak q)$ obeys the same commutation relation as $\mathcal A$: that is, $Q_1(\frak q)^\star \sigma_D=\sigma_D^\star Q_1(\frak q)$. The second main result is the following corollary, whose proof occupies Section \ref{section3}.

\begin{cor}\label{cornonlinear} If $\frak q$ solves the non-linear equation \be (D+ \tfrac{1}{\e}\mathcal A)\frak q + Q(\frak q)=f \label{nonlinearDirac} \ee 

\noindent where $D,\mathcal A$ are as in Theorem \ref{maina}, $f\in \Gamma(\frak H^\perp)$, and where $Q(\frak q)=Q_1(\frak q)\pi_{\frak H}(q)$ is of the above form and satisfies \be \e\|Q_1(\frak q)\|_{L^{1,n}(K')} \to 0 \hspace{2cm} \e\|Q_1(\frak q)\|_{C^0(K')}\to 0.\label{quadraticestimates}\ee

\noindent Then the conclusion  (\refeq{expdecay}) of Theorem \ref{maina} holds. 
\end{cor}

\medskip 

The motivation for and application of the above results come from investigations of the compactness of moduli spaces of solutions to generalized Seiberg-Witten equations (see Sections Sections \ref{section4.1}--\ref{section4.2} below). As explained in the introduction, it is known in many cases of interest that sequences of solutions to (\refeq{prelimSW1}--\refeq{prelimSW2}) lacking bounded $L^2$ subsequences subconverge after renormalization to a type of $\Z_2$-harmonic spinor. There are four cases of interest: 
\medskip

{\bf Case (I):} The two-spinor Seiberg-Witten equations on a compact 3-manifold $Y$  \cite{HWCompactness} 

{\bf Case (II):} The equations for a flat $\text{SL}(2,\C)$ connection on a compact 3-manifold $Y$\cite{Taubes3dSL2C,WalpuskiZhangCompactness}  

{\bf Case (III):} The two-spinor Seiberg-Witten equations on a compact 4-manifold  $X$ \cite{TaubesU1SW} 

{\bf Case (IV):} The $\text{SL}(2,\C)$ ASD equations on a compact 4-manifold $X$  \cite{Taubes4dSL2C}

\medskip
\noindent  The precise statements of the known compactness results for these cases is amalgamated in Theorem \ref{compactness} in Section \ref{section4.2}. A precise definition of $\Z_2$-harmonic spinors is given in the subsequent Definition \ref{Z2Harmonic}. 

\begin{thm}\label{mainb} Suppose that a sequence $(\Phi_i, A_i, \e_i)$ of re-normalized solutions to a generalized Seiberg-Witten equation converge to a $\Z_2$-harmonic spinor $(\mathcal Z_0, A_0, \Phi_0)$ in the sense of Theorem \ref{compactness} for one of Cases (I)--(IV) above. Then, in the limit $\e_i\to 0$ 

\medskip

\noindent {Case (I)--(III):} The equations (\refeq{prelimSW1}--\refeq{prelimSW2}) behave as a non-linear concentrating Dirac operator satisfying  \indent \indent \indent  \indent  \   \ \  \ \ the assumptions of Corollary \ref{cornonlinear}.
\smallskip

\hspace{.05cm}{Case (IV)}  :     If, in addition to the conclusion of Theorem \ref{compactness}, the sequence $(\Phi_i,A_i)$ satisfies   $A_i\to A_0$ \indent\indent   \indent \indent \ \  \ \  \ in $L^{1,p}_{loc}$ and $\Phi_i \to \Phi_0$ in $L^{2,p}_{loc}$ for any $p>2$, then the same conclusion holds. 

\medskip 

Additionally, in these situations $$\Phi_i\overset{C^\infty_\text{loc}} \lre \Phi_0 \hspace{2cm} A_i \overset{C^\infty_\text{loc}}\lre A_0$$ where local convergence mean on compact subsets $K\subseteq Y-\mathcal Z_0$ (resp. $X-\mathcal Z_0$ in Cases (III)--(IV)).  Moreover, in all cases the limiting connection satisfies $F_{A_0}=0$ (resp. $F_{A_0}^+=0$). 
\end{thm}

\medskip

\begin{rem}
In Cases (I)--(III), Theorem \ref{mainb} applies directly to strengthen the results of \cite{Taubes3dSL2C,HWCompactness,TaubesU1SW,WalpuskiZhangCompactness}.  In Case (IV), the additional requirement beyond the
conclusions of \cite{Taubes4dSL2C} that $A_i\to A_0$ in $L^{1,p}_{loc}$ and $\Phi_i\to \Phi_0$ in $L^{2,p}_{loc}$ for some $p>2$ is necessary. This is needed to overcome the quadratic non-linearity in $F_A^+$ which is borderline for the relevant Sobolev embeddings in dimension 4 (this assumption is not necessary in Case (III), which is abelian). It is unclear if this slightly stronger convergence result than Theorem \ref{compactness} can be proved by extending the techniques of \cite{Taubes4dSL2C, TaubesVW}, or if a new approach is required to bridge the gap from the results therein to the point where Theorem \ref{mainb} applies. 

\end{rem}

\begin{rem}  \ 
\cite{TaubesVW} proves a version of Theorem \ref{compactness} for the Vafa-Witten equations as well, though this case is notably absent from Theorem \ref{mainb}. Although the equations have a similar form to the equations of Cases (I)--(IV), it is not expected for the Vafa-Witten equations that the sequence of connections $A_i$ should converge. In particular, for these equations there exist sequences of solutions such that the $L^2$-norm of the curvature over any compact set with non-empty interior diverges. Solutions with this behavior are the subject of forthcoming work of C. Taubes \cite{TaubesVWCounterexamples}, and were known to E. Witten and others prior to the work \cite{TaubesVW}. For sequences along which the connections $A_i$ do happen to converge in $L^{1,p}_{loc}$ for some $p>2$, then the conclusions of Theorem \ref{mainb} apply---the proof is trivially different from Case (IV) and is omitted.    
\end{rem}

\begin{rem}\label{rem1.6}
Also absent from the cases enumerated in Theorem \ref{mainb} are the compactness theorems for the $\text{ADHM}_{1,2}$ Seiberg-Witten equations established in \cite{WalpuskiZhangCompactness} and for Seiberg-Witten equations with $r>2$ spinors from \cite{TaubesU1SW, HWCompactness}. The approach of Theorem \ref{mainb} fails in these cases, because the $\mathfrak H$-components of the equations are too strongly coupled to the other components. This occurs in two slightly different ways:  for the $\text{ADHM}_{1,2}$ Seiberg-Witten equations, the linearized equations properly fit into the framework of Theorem \ref{maina}, but the non-linear terms do not satisfy the hypotheses of Corollary \ref{cornonlinear}. For the Seiberg-Witten equations with $r>2$ spinors, the reverse occurs: the non-linear terms have the desired form, but the splitting \refeq{blockdiagonal} fails to be parallel and the $\mathfrak H$-components are too strongly coupled to the other components for the linearized equations. It is an interesting task to investigate whether Theorem \ref{maina} may be extended to the case where the splitting is not parallel. 
\end{rem}

\bigskip 
\section*{Acknowledgements}
This work grew out of the authors Ph.D. thesis. The author is grateful to his advisor Clifford Taubes for his support and suggestions. This work also benefitted from conversations with Rafe Mazzeo, Aleksander Doan, and Thomas Walpuski, and was supported by a National Science Foundation Graduate Research Fellowship and by National Science Foundation Grant No. 2105512. 

\bigskip 
\section{Concentrating Dirac Operators}

\label{section2}

Let $(Y,g)$ be an n-dimensional Riemannian manifold (not necessarily compact), and $E\to Y$ a real Clifford module. That is to say, $E$ is a real vector bundle equipped with i) an inner product, denoted $\br -,-\kt$, ii)  a metric-compatible connection denoted $\nabla$, and iii) a Clifford 
multiplication $\sigma_D: T^*Y\to \text{End}(E)$ obeying the Clifford relation 
\be  \sigma_D(\xi) \sigma_D(\eta) + \sigma_D(\xi) \sigma_D(\eta) =-2\br \xi, \eta\kt. \ee

\noindent If $n$ is even, we assume that there is a splitting $E=E^+\oplus E^-$ with respect to which $\sigma_D$ is off-diagonal. Consider an $\e$-parameterized family of Dirac operators $D_\e: \Gamma(E)\to \Gamma(E)$ of the form 
\begin{equation}
D_\e \mathfrak q= \left(D+ \frac{1}{\e}\mathcal A\right)\mathfrak q
\end{equation}
where $D$ is the Dirac operator $\sigma_D \circ \nabla$, and $\mathcal A:E\to E$ is a bundle map. In the case that $n$ is even, we denote by the same symbol the Dirac operator $D: \Gamma(E^+)\to \Gamma(E^-)$ and assume that $\mathcal A: E^+\to E^-$. The operators $D_\e$ and $\mathcal A$ are only assumed to be $\R$-linear, even in the case that $E$ possesses a complex structure. 

Let $r$ denote the maximal rank of $\mathcal A$, and define the {\bf {singular set}} of $\mathcal A$ by $$\mathcal Z:= \{ y\in Y \ | \ \text{rank}(\mathcal A(y))<r\}.$$  
That is, the set where $\mathcal A$ has strictly lower than its maximal rank. We assume that $\mathcal A$ is continuous, hence $\mathcal Z$ is closed. We additionally assume, for convenience, that $\mathcal Z$ is non-empty (else all assertions are vacuous). In the case of the Seiberg-Witten equations linearized at a $\Z_2$-harmonic spinor $\Phi_0$, this set coincides with the singular set $\mathcal Z=|\Phi_0|^{-1}(0)$ of the $\Z_2$-harmonic spinor. We do NOT assume that $\mathcal A$ is smooth (and indeed this is not the case in the setting of $\Z_2$-harmonic spinors), nor do we assume that $\mathcal Z$ is a submanifold or a union of submanifolds.

The focus will be on the following class of operators. 
\begin{defn}\label{concentratingdiracdef}
A perturbed Dirac operator $$D_\e= D+\tfrac{1}{\e}\mathcal A$$ is said to be a  {\bf{concentrating Dirac operator with fixed degeneracy}} if it obeys the following two properties: 

\begin{enumerate}
\item[1).] {\bf {(Concentration property)}} The principal symbol $\sigma_D$ of $D$ obeys \be\mathcal A^\star  \sigma_D(\xi) = \sigma_D(\xi)^\star  \mathcal A \hspace{2cm}\forall \  \xi\in T^\star Y\label{concentrationproperty}\ee
\item[2).] {\bf (Fixed Degeneracy)}: On $Y-\mathcal Z$, there is a bundle splitting $E|_{Y-\mathcal Z}= \mathfrak N \oplus \mathfrak H$ (resp. $E^\pm|_{Y-\mathcal Z} = \mathfrak N^\pm \oplus \mathfrak H^\pm$ if $n$ is even) which is parallel with respect to $\nabla$ and preserved by $\sigma_D$, and in this splitting $\mathcal A$ takes the form \be\mathcal A=\begin{pmatrix}
0 & 0 \\ 0 & A_{\mathfrak H} \end{pmatrix}\label{fixeddegeneracy}\ee 
where $A_{\mathfrak H}: \mathfrak H\to \mathfrak H$ (resp. $A_\mathfrak H: \mathfrak H^+\to \mathfrak H^-$).   \end{enumerate}

\end{defn}

Note that item 2) implies that $\frak N=\ker(\mathcal A)$. By taking $\mathfrak N$ to be empty, this definition generalizes the class of operators previously considered in \cite{ManosConcentrationI,ManosConcentrationII, RussianConcentratingIndex} where $\mathcal A$ is invertible on an open dense set. In Section \ref{section4}, it is shown that generalized Seiberg-Witten equations provide a rich class of examples for which the degeneracy is non-trivial.

\bigskip 

The remainder of this section establishes Theorem \ref{concentrationprinciple}. To begin, we prove a general Weitzenb\"ock formula for concentrating Dirac operators with fixed degeneracy. 
\begin{lm} A concentrating Dirac operator $D_\e$ satisfies $$D_\e ^\star D_\e = D^\star D + \frac{1}{\e^2} \mathcal A^\star \mathcal A + \frac{1}{\e} \mathfrak B$$ where $\mathfrak B$ is a zeroeth order term. 
\label{weitzenbockabstract}
\end{lm}

\begin{proof}
The key point is that the concentration property (\refeq{concentrationproperty}) ensures that the first-order cross terms cancel.  
\bea
D_\e ^\star D_\e\mathfrak q 
&=& D^\star D \mathfrak q+ \tfrac{1}{\e^2}\mathcal A^\star  \mathcal A \mathfrak q + \tfrac{1}{\e}\left(D^\star  \mathcal A+ \mathcal A^\star  D\right)\mathfrak q\\
&=&  D^\star D \mathfrak q+ \tfrac{1}{\e^2}\mathcal A^\star \mathcal A \mathfrak q + \tfrac{1}{\e}\left(\sigma_D(e^j)^\star \nabla_j^\star   (\mathcal A\mathfrak q)+ \mathcal A^\star \sigma_D(e^j) \nabla_j\mathfrak q \right)\\
&=& D^\star D \mathfrak q+ \tfrac{1}{\e^2}\mathcal A^\star  \mathcal A \mathfrak q + \tfrac{1}{\e}\big(-\sigma_D(e^j)^\star  (\nabla_j  \mathcal A)\mathfrak q + {\cancel {[- \sigma_D(e^j)^\star   \mathcal A + \mathcal A^\star \sigma_D(e^j)]}}\nabla_j\mathfrak q \big)\\
&=& D^\star D \mathfrak q+ \tfrac{1}{\e^2}\mathcal A^\star  \mathcal A \mathfrak q + \tfrac{1}{\e}\left(-\sigma_D(e^j)^\star (\nabla_j \mathcal A)\mathfrak q \right)\\
&=& D^\star D\mathfrak q +\tfrac{1}{\e^2}\mathcal A^\star \mathcal A\mathfrak q +\tfrac{1}{\e}\mathfrak B\mathfrak q
\eea
where the last line is taken as the definition of $\mathfrak B$. 
\end{proof}

\bigskip

We now prove Theorem \ref{concentrationprinciple}.  The proof requires two brief lemmas. For these, we denote the components of a spinor $\frak q$ in the splitting of (\refeq{fixeddegeneracy}) by 
 $$\mathfrak q=(q_0,q_1) \in \mathfrak N\oplus \mathfrak H.$$ 
 \noindent Additionally, we fix a compact subset $K\Subset Y-\mathcal Z$. To avoid separating cases, it is understood in the above expression, and in the discussion that follows that in the even-dimensional case the bundles $\mathfrak N\oplus \mathfrak H$ undecorated by superscripts refer to $\mathfrak N^+$ and  $\mathfrak H^+$.

\begin{lm} If $D_\e \frak q=0$, then for sufficiently small $\e$ the scalar quantity $|q_1|^2$ satisfies the differential inequality \be d^\star d |q_1|^2 +\frac{1}{\e^2}|\mathcal Aq_1|^2\leq 0 .\label{diffinequality3.2}\ee 
\end{lm}

\begin{proof}
For the $\e$-independent operator $D$, one has a Weitzenb\"ock formula \be D^\star D= \nabla ^\star \nabla + \mathcal R \label{noeweitzenbock}\ee where $\mathcal R$ is Clifford multiplication by a curvature term which is bounded in $C^0$ by a constant independent of $\e$. The fixed degeneracy assumption (\refeq{fixeddegeneracy}) implies that $(0,q_1)\in \ker(D_\e)$. In a slight abuse of notation, we denote this pair simply by $q_1$, so that 

\be  D_\e^\star  D_\e q_1=0.\label{parallelD}\ee
\noindent  Using (\refeq{noeweitzenbock}),  (\refeq{parallelD}), and Lemma \ref{weitzenbockabstract}, for $\e$ sufficiently small we have  
\bea -\frac{1}{2} d^\star d |q_1|^2  &=&  \br \nabla q_1 , \nabla q_1 \kt + \br q_1 , -\nabla^\star  \nabla q_1 \kt \\
&=& |\nabla q_1|^2 + \br q_1 , \mathcal Rq_1 \kt  + \br q_1 , -D^\star D q_1 \kt\\
&=& |\nabla q_1|^2 + \br q_1 , \mathcal Rq_1 \kt  + \tfrac{1}{\e^2}|\mathcal Aq_1|^2 +\tfrac{1}{\e} \br q_1 , \mathfrak B q_1 \kt\\
&\geq &  \tfrac{1}{2\e^2}|\mathcal Aq_1|^2. 
 \eea  
 In the last line, we have used the fact that $\mathcal R$ and $\mathfrak B$ are bounded independent of $\e$, so can be absorbed for sufficiently small $\e$ since the injectivity of $\mathcal A$ on $\frak H |_{Y-\mathcal Z_0}$ implies that $|\mathcal A q_1|> c|q_1|$ holds for some constant $c$ on a compact subset $K\Subset Y-\mathcal Z_0$.
\end{proof}

\bigskip 

 The next lemma is an abstract result about scalar functions satisfying differential inequalities of the form (\refeq{diffinequality3.2}). The statement uses the following notation: for each $y\in Y-\mathcal Z$, let $\Lambda(y)= \inf_{\|v\|=1} \|Av\|$ where $v\in \mathfrak H_y$. The fixed degeneracy hypothesis (\refeq{fixeddegeneracy}) ensures $A_\mathfrak H$ is invertible on $Y-\mathcal Z$, hence $\Lambda(y)>0$.  
 Then, for a fixed compact subset $K\Subset Y-\mathcal Z$ as in the statement of Theorem \ref{concentrationprinciple}, let  $\Lambda_K:=\inf_{y\in K}\Lambda(y)$ so that
$$|\mathcal A \frak q(y)|^2\geq \Lambda(y)^2 |q_1(y)|^2\geq \Lambda^2_K|q_1(y)|^2 $$ 

\noindent holds for $y\in K$. As in the statement of Theorem \ref{maina}, $R_K=\text{dist}(K,\mathcal Z)$. Additionally, in what follows, $\Delta_g =d^\star d$ denotes the positive definite Laplacian defined by the Riemannian metric $g$. 

\begin{lm}
Suppose that $u: K \to \R$ satisfies $u\geq0$ and \be \Delta_g u+  \frac{\Lambda(y)^2}{\e^2} u \leq 0. \label{diffinequality}\ee
Then there exist constants $C,c>0$ such that at a point $y_0\in K$ one has 

\be u(y_0)\leq \frac{C}{R_K^n}\text{Exp}\left( -\frac{\Lambda(y_0)c}{\e}R_K\right) \int_{K'}  |du|+ {|u|} \ dV\label{resultofgreensfunctionlem}\ee
\label{greensfunction}
\noindent where $K'$ is a compact set with $\overset{\circ}K\Subset K'\Subset Y-\mathcal Z$. 
\end{lm}

\begin{proof}

Recall first that the Green's function of $\Delta+{m^2}$ on $\R^n$ is 
\begin{equation}
G(x,x_0)= \frac{C(n)}{ |x-x_0|^{n-2}} \text{Exp}\left({-{m}|x-x_0|}\right)\label{RiemannianGreens}
\end{equation}
 where $C(n)$ is a constant depending only on the dimension. Let $K' \Subset Y-\mathcal Z$ be a compact set whose interior contains $K$, and take  $1>c_0>0$ to be a small number such that the following three conditions hold for $R_0= c_0R_K$: 
\begin{enumerate}
\item[(i)] $y\in K \Rightarrow B_{R_0} \subseteq K'$
\item[(ii)] $y\in B_{R_0}(y_0) \Rightarrow \Lambda(y)^2 \geq \tfrac{\Lambda_0^2}{2}$ where $\Lambda_0=\Lambda(y_0)$. 
\item[(iii)] The Green's function of $-\Delta_g-\tfrac{\Lambda_0^2}{2\e^2}$ on $B_{2R_0}$ with Dirichlet boundary conditions satisfies \bea G(y,y_0)&\leq& \frac{c_1}{|y-y_0|^{n-2}}\text{Exp}\left({-\frac{\Lambda_0}{2\e}|y-y_0|}\right)\eea

\noindent uniformly on $Y$ (resp. $X$) once $\e$ is sufficiently small (depending on $K$), and where $c_1$ depends only on the dimension.  
\end{enumerate}

\noindent The first of these is possible by the compactness of $K$, the second by the fact that $\mathcal A$ is $C^1$. The third follows by using the comparison principle on $B_{2R_0}$ comparing with the Green's function (\refeq{RiemannianGreens}) on Euclidean space. The details are provided in Appendix \ref{Greensappendix}.

Now let $y_0\in K$ and consider the ball $B_0= B_{R_0}(y_0)\subseteq K$, where $R_0$ is as above. Green's identity on $B_0$ for functions $\eta,\psi$ states that $$\int_{B_0} \eta(- \Delta_g \psi) - \psi(- \Delta_g \eta) \ dV_g= \int_{\del B_0} \eta \del_\nu \psi - \psi \del_\nu \eta  \ dS_g$$ 
and is derived by integrating the quantity $0=\int \br d\eta, d\psi \kt-\br d\psi, d\eta\kt$ by parts. Here, $dV_g, dS_g$ denote the volume forms arising from the Riemannian metric on the ball and the sphere respectively. Now set $$M:={\frac{\Lambda_0}{\sqrt{2}\e}}.$$ Adding and subtracting $M^2 \eta \psi$ on the left-hand side yields 

\be\int_{B_0} \eta(- \Delta_g -M^2) \psi + \psi( \Delta_g +M^2) \eta \ dV_g= \int_{\del B_0} \eta \del_\nu \psi - \psi \del_\nu \eta  \ dS_g.\label{GreensIdentity}\ee

Next, we let $\beta(y)$ denote a cutoff function equal to $1$ on $B_{R_0/2}(y_0)$ and supported in the interior of  $B_{0}$ satisfying \be 0\leq \beta \leq 1 \hspace{2cm} |\nabla \beta|\leq \frac{C}{R_0} \hspace{2cm} |\nabla^2\beta|\leq \frac{C}{R_0^2}\label{cutoffbounds}\ee and apply the identity \refeq{GreensIdentity} with $\eta=u(y)\beta(y)$ and $\psi= G(y,y_0)$. The boundary term on the right hand side vanishes by the choice of $\beta$. The first term becomes $-u(y_0)$ since, by definition, the Green's function satisfies $(\Delta_g + M^2)G=\delta_{y_0}$. Meanwhile for the second term, the assumption that $u$ satisfies(\refeq{diffinequality}) implies  \be(\Delta_g +M^2)u=\left (\Delta_g +\frac{\Lambda_0^2}{2\e^2}\right)u \leq \left(\Delta_g+ \frac{\Lambda(y)^2}{\e^2}\right)u\leq 0 \label{diffineq} \ee

\noindent hence
$$(\Delta_g +M^2) \beta u= \beta (\Delta_g +M^2)u - 2 \br d\beta, du\kt+ (\Delta_g \beta)u \leq C\left(\frac{1}{R_0}|du| + \frac{1}{R_0^2}|u|\right) \chi_{A}  $$

\noindent where $\chi_{A}$ is the characteristic function equal to 1 on the outer annulus $A=\{R_0/2\leq r \leq R_0\}$ and vanishing elsewhere. 

Thus the identity (\refeq{GreensIdentity}) becomes the inequality
\bea
u(y_0)&\leq&C  \int_{A_{}} G(y,y_0)\left(\frac{1}{R_0}|du| + \frac{1}{R_0^2}|u|\right)  \ dV_g\\
& \leq & \frac{C}{R^{n-1}_0}\text{Exp}\left(-\frac{\Lambda_0}{8\e} R_0\right) \int_{A_{}} |du|+ \frac{|u|}{R_0}. \\
\eea
Substituting the definition $R_0= c_0 R_K$ yields the bound (\refeq{resultofgreensfunctionlem}). 
\end{proof}

\begin{proof}[Proof of Theorem \ref{concentrationprinciple}] Apply Lemma \ref{greensfunction} to $u=|q_1|^2$. Kato's inequality $d|q| \leq |\nabla q|$ shows that $$d|q_1|^2 \leq 2 |q_1| d|q_1| \leq 2|q_1| |\nabla q_1| \leq |q_1|^2 + |\nabla q_1|^2.$$ 

\noindent Applying this to the right side of (\refeq{resultofgreensfunctionlem}) and setting $R_K=c_0R_0$ yields
\begin{eqnarray} |q_1(y_0)|^2 &\leq &  \frac{C}{ c_0^n R_K^n } \text{Exp}\left(-\frac{\Lambda_0c_0}{\e} R_K\right) \|q_1\|^2_{L^{1,2}(K')}.  \label{rrrgh}\end{eqnarray}
 
 \noindent Taking the square root and decreasing constants by a factor of $2$ to replace $\Lambda_0$ by $\Lambda_K$ shows the desired estimate (\refeq{expdecay}).

\end{proof}

\bigskip 

\section{Non-Linear Concentrating Dirac Equations}
\label{section3}
This section extends the proof of Theorem \ref{concentrationprinciple} to the case of the non-linear equation (\refeq{nonlinearDirac}) of Corollary \ref{cornonlinear}.

 Let us clarify the subtlety in deducing the estimate (\refeq{expdecay}) for a non-linear Dirac equation. Equation (\refeq{nonlinearDirac}) can be written as the solution of a concentrating Dirac operator as follows: let  $A_\e(q_1)=\e Q_1(\frak q)q_1$ with $Q_1$ as in the statement of Corollary \ref{cornonlinear}. Then 
 
  \begin{eqnarray} D_\e\frak q+Q(\frak q)&=& 0 \\ \left( D+\tfrac{1}{\e}\mathcal A + Q_1(\frak q)\right) \frak q&=&0  \label{nonlinearQ11}\\
 \left( D+\tfrac{1}{\e}\mathcal (A_\mathfrak H + A_\e(\frak q))\right)\frak q &=& 0  \label{nonlinearQ12}\end{eqnarray}
i.e. $\frak q$ is the solves a concentrating Dirac equation, but the zeroth order perturbation now depends on $\frak q$. In this case, the proof  of Theorem \ref{concentrationprinciple} fails in general because it is no longer clear we can absorb the $\mathfrak B=\nabla \mathcal A$ term. Indeed, this would require a $C^1$-bound on $A_\e$. The naive bootstrapping, however, leads only to bounds with powers of $\e^{-1}$ larger than can be absorbed (see also Section \ref{section5}). The proof of Corollary \ref{cornonlinear} relies on stronger differential inequalities obtained by not discarding the $|\nabla \frak q_1|^2$ term in Lemma \ref{diffinequality3.2}.

Corollary \ref{cornonlinear} is deduced from the following proposition by setting $A_\e(-)=\e Q_1(\frak q,-)$. The additional assumption on $\e Q_1$ now manifests as the requirement that the concentration property $A_\e^\star \sigma_D-\sigma_D^\star A_\e$ is satisfied for this term as well.

   \begin{prop}\label{perturbationcase} The conclusion (\refeq{expdecay}) of Theorem \ref{concentrationprinciple} continues to hold if the zeroth order term is given by 
\be\mathcal A=\begin{pmatrix}
0 & 0 \\ 0 & A_{\mathfrak H}\end{pmatrix}  + A_\e .\label{fixeddegeneracy2}\ee

\noindent where $A_{\mathfrak H}$ is smooth with all derivatives bounded independent of $\e$ on $K$ and  $\frak q, A_\e$ satisfy the relation 
 $A_\e^\star \sigma_D(\xi)=\sigma_D^\star(\xi) A_\e=$ for $\xi \in T^\star Y$ and 

$$\|A_\e\|_{L^{1,n}(K')}\to 0 \hspace{2cm} \|A_\e\|_{C^0(K')}\to 0. $$

\noindent More generally, the same conclusion holds for the inhomogeneous version of equation $(\refeq{nonlinearQ12})$, 

$$ \left( D+\tfrac{1}{\e}\mathcal (A_\mathfrak H + A_\e(\frak q))\right)\frak q =f$$
 where $f$  satisfies $\pi_{\frak H}(f)=0$. 
\end{prop}

\begin{proof}

Retaining the notation of Lemma \ref{greensfunction}, let $\Lambda_K=\inf_{y\in K}\Lambda(y_0)$ where $\Lambda(y)= \inf_{\|v\|=1} \|Av\|$ for $v\in \mathfrak H_y$ be as before (thus it has no dependence on $A_\e$). Additionally, write $$\mathfrak B= \nabla A_{\mathfrak H} + \nabla A_\e$$

\noindent where $\nabla$ is shorthand for the operator $\sigma_D(e^j)^\star \nabla_j $ appearing in the proof of Lemma \ref{weitzenbockabstract}. 

Proceeding as in the proof of Lemma \refeq{diffinequality3.2}, we have the following differential inequality. Note here that the inner product with $q_1$ and the assumption that $\pi_\frak H(f)=0$ imply that the cases of $f=0$ and $f\neq 0$ yield the same inequality, and the two cases therefore coincide for the remainder of the proof.  

\bea -\frac{1}{2} d^\star d |q_1|^2 
&=& |\nabla q_1|^2 + \br q_1 , \mathcal Rq_1 \kt  + \tfrac{1}{\e^2}|\mathcal Aq_1|^2 +\tfrac{1}{\e} \br q_1 , \mathfrak B q_1 \kt\\
&\geq & |\nabla q_1|^2 + \br q_1 , \mathcal Rq_1 \kt  + \tfrac{1}{\e^2}|(A_{\mathfrak H}+ A_\e)q_1|^2 +\tfrac{1}{\e} \br q_1 , (\nabla A_{\mathfrak H} +\nabla A_\e)q_1 \kt \\ 
&\geq & |\nabla q_1|^2 + \br q_1 , \mathcal Rq_1 \kt  + \tfrac{1}{\e^2}|(A_{\mathfrak H}+ A_\e)q_1|^2 +\tfrac{1}{\e} \br q_1 , (\nabla A_{\mathfrak H} + \nabla A_\e)q_1 \kt \\ 
&\geq & |\nabla q_1|^2 + \tfrac{1}{2\e^2}|A_{\mathfrak H}q_1|^2 +\tfrac{1}{\e}\br q_1 , (\nabla A_\e) q_1\kt  
 \eea
 
 \noindent where in addition to the absorptions done before, we have used that $\|A_\e\|_{C^0(K')}\to 0$ which implies that \be |A_{\e} q_1|^2 \leq \frac{\Lambda_K^2}{4} |q_1|^2 \leq \frac{1}{4}|A_{\mathfrak H} q_1|^2\label{absorbC0}\ee
 \noindent once $\e$ is sufficiently small. Rearranging shows the differential inequality now dictates that 
 
  \be\Delta |q_1|^2 + \frac{|A_{\mathfrak H} q_1|^2}{\e^2} \leq -2 |\nabla q_1|^2 - \tfrac{2}{\e}\br q_1,( \nabla A_\e) q_1\kt \label{improveddifferential}\ee
  
  \noindent on $K'$. 
  
 Proceeding as before (but with an additional factor of $\sqrt{2}$), set $M:={\frac{\Lambda_0}{2e}}$. Taking $u=|q_1|^2$, in this case (\refeq{GreensIdentity}) becomes

  \begin{eqnarray}
u(y_0)&=&- \int_{B_0} G(y,y_0) (-\Delta_g - M^2) (\beta u) \ dV_g  \nonumber  \\
  &\leq  &  \int_{B_0} G(y,y_0) \beta (\Delta_g + M^2)u) \ dV_g  +  C  \int_{B_{0}} G(y,y_0)\left(\frac{1}{R_0}|du| + \frac{1}{R_0^2}|u|\right)  \ dV_g   \nonumber 
  \end{eqnarray}
  
  \noindent and combining this with \refeq{improveddifferential} yields

  \begin{eqnarray}\hspace{-1cm} |q_1|^2(y)  & +&   \int_{B_0} G(y, y_0)\beta\left(  |\nabla q_1|^2 + \frac{\Lambda_0}{8\e^2}|q_1|^2 \right) \ dV_g  \label{3.3}\\  &\leq & C  \int_{B_{0}} G(y,y_0)\left(\frac{1}{R_0}|d|q_1|^2| + \frac{1}{R_0^2}|q_1|^2\right)  \ dV_g +\frac{1}{\e} \int_{B_0}  G(y_0, y)\beta |\br q_1 , (\nabla A_\e) q_1\kt| dV_g.\label{tointerpolate} \end{eqnarray}
  
    \medskip

  \medskip Using Lemma \ref{interpolationclaim} below, to absorb the second term on the right in (\refeq{tointerpolate}) into the two remaining terms of (\refeq{3.3})--(\refeq{tointerpolate}) leads to
  \begin{eqnarray} |q_1|^2(y) &\leq & 2C  \int_{B_{0}} G(y,y_0)\left(\frac{1}{R_0}|d|q_1|^2| + \frac{1}{R_0^2}|q_1|^2\right)  \ dV_g \end{eqnarray}
 after which the result follows identically to Theorem \ref{concentrationprinciple} by repeating the argument leading to (\refeq{rrrgh}).
\end{proof}

 \begin{lm}\label{interpolationclaim}
  For $\e$ sufficiently small, the integral $$I=  \frac{1}{\e} \int_{B_0}  G(y_0, y)\beta |\br q_1 , (\nabla A_\e) q_1\kt| dV_g $$
  
 \noindent  satisfies 
  \begin{eqnarray}
 I & \leq  & \int_{B_0} G(y, y_0)\beta\left(  |\nabla q_1|^2 + \frac{\Lambda_0}{8\e^2}|q_1|^2 \right)  dV_g  + C  \int_{B_{0}} G(y,y_0)\left(\frac{1}{R_0}|d|q_1|^2| + \frac{1}{R_0^2}|q_1|^2\right)  dV_g \label{tointerpolate2} 
  \end{eqnarray}
  
  \end{lm}

\begin{proof}[Proof of Lemma \ref{interpolationclaim}] This follows from a weighted interpolation inequality and a dyadic decomposition. Let $$A_n=\big\{ |y| \in [r_{n+1},r_n]\big\}$$
for $n\geq 0$ be a sequence of disjoint annuli covering $B_0$ where the radii $r_n$ are defined inductively by  
\bea
r_0&=&R_0\\
r_{n+1}&=& r_n - \min \left\{{\tfrac{r_n}{5}, \tfrac{1}{M}}\right\}.
\eea
Next, subdivide each $A_n$ into a disjoint union of sectors $A_{n\ell}$ such that $\text{diam}(A_{n\ell})\leq |r_{n+1}-r_n|$. It is now enough to prove (\refeq{tointerpolate2}) separately for each of the disjoint sectors $A_{n\ell}$.

  On these sectors, we have the following weighted interpolation inequality: 
  If $p$ satisfies 
  
  \be \frac{1}{p}= \frac{j}{n}+ \alpha\left(\frac{1}{r}-\frac{m}{n}\right)+ \frac{1-\alpha}{s}. \label{**}\ee
  then the following inequality holds uniformly over the collection $A_{n\ell}$ and uniformly in $M$ (on which $G=G(y,y_0)$ depends) . 
\begin{eqnarray}  \left( \int_{A_{n\ell}}G^{p/2} |\nabla^j v|^{p} \ dV \right)^{1/p}&\leq&  C \left( \int_{A_{n\ell}}G^{r/2}|\nabla^m v|^r \ dV \right)^{\alpha/r} \left( \int_{A_{n\ell}}G^{s/2}|v|^s \ dV \right)^{(1-\alpha)/s} \label{AnlinterpolationI} \nonumber \\ & & + C \left( \int_{A_{n\ell}}G| v|^2 \ dV \right)^{1/2}.\label{AnlinterpolationII}  \end{eqnarray}

  To prove (\refeq{AnlinterpolationII}), notice that without the Green's function weight (ie setting $G=1$) this follows from the scale invariance of the standard interpolation inequalities (the $L^2$-term gets comparatively stronger on smaller balls). The version including $G$ follows from this and invoking the Harnack inequality (\cite{GilbargTrudinger}, Theorem 8.20)  on each $A_{n\ell}$ which shows $$\sup_{A_{n\ell}} |G(y_0,y)|  \leq C \inf_{A_{n\ell}} |G(y_0,y)|.$$
  
  \noindent Indeed, by construction, each $A_{n\ell}$ is contained in a ball of radius $R_n$ such that the ball of radius $4R_n\subseteq B_{R_0}-\{y_0\}$. Since $G(y_0,y)\geq 0$ by the maximum principle and satisfies $(\Delta_g + M^2) G(y_0,y)=0$ on $B_{R_0}-\bigcup_{k\geq n+3} A_k$, the Harnack inequality  (\cite{GilbargTrudinger}, Theorem 8.20) applies. Moreover, by the scaling of the constant (see \cite{GilbargTrudinger} Theorem 8.20 and the subsequent comments), the constant can be taken to be uniform since $r_n\leq \tfrac{1}{M}$ by construction (so that in the notation of 8.20 one has $\nu R=O(1)$).
  
  \bigskip

  We conclude the lemma using (\refeq{AnlinterpolationII}). Without disrupting the bounds (\refeq{cutoffbounds}) we may assume that $\beta=\chi^2$ is the square of another smooth cut-off function satisfying the same bounds up to universal constants.  In the case of dimension 4, applying H\"older's inequality with with $q^*=4$ and $p^*=4/3$ yields  
  \begin{eqnarray}
  \frac{1}{\e} \int_{A_{n\ell}}  G(y_0, y)\beta |\br q_1 , (\nabla A_\e) q_1\kt| dV_g & \leq& \frac{1}{\e} \|A_\e\|_{L^{1,4}(B_0)}   \left( \int_{A_{n\ell}}G^{p/2} |  q_1 \sqrt{\beta}|^{p} \ dV \right)^{2/p}\label{inequality1}
  \end{eqnarray}
  where $p=8/3$. Next, we apply the interpolation inequality (\refeq{AnlinterpolationII}) to $q_1 \sqrt{\beta}$ with $r=s=2$, and $j=0$ and $m=1$ in which case (\refeq{**}) shows $\alpha=\frac{1}{2}$
hence (\refeq{inequality1}) is bounded by 

\bea
&\leq & \frac{C}{\e}\|A_\e\|_{L^{1,4}(B_0)}  \left[\left( \int_{A_{n\ell}}G|\nabla (\chi q_1)|^2 \ dV \right)^{1/2} \left( \int_{A_{n\ell}}G|\chi q_1|^2 \ dV \right)^{1/2}  +C\| G^{1/2} \chi q_1\|^2_{L^2(A_{n\ell})}\right]\\
 &\leq & \frac{C}{\e}\|A_\e\|_{L^{1,4}(B_0)}  \left( \e\frac{\|G^{1/2}\nabla (\chi q_1)\|^2_{L^2(A_{n\ell})}}{2} +  \frac{\|G^{1/2} \chi q_1\|^2_{L^2(A_{n\ell})}}{2\e}  +\| G^{1/2} \chi q_1\|^2_{L^2(A_{n\ell})}\right) \\
 &\leq & {C}\|A_\e\|_{L^{1,4}(B_0)}  \left( \int_{A_{n\ell}} G(y, y_0)\beta\left(  |\nabla q_1|^2 + \frac{1}{2\e^2}|q_1|^2 \right) \ dV   +\int_{A_{n\ell}} G(y_0, y)\frac{1}{R_0^2}|q_1|^2\ dV \right) 
\eea

\noindent 
where we have used the bounds  (\refeq{cutoffbounds}) on $d\chi$ to obtain the last term. The assumption that $\|A_\e\|_{L^{1,4}}\to 0$ allows us to choose $\e$ sufficiently small that $\left(C+\tfrac{C\Lambda_0}{8}\right)\|A_\e\|_{L^{1,4}}\leq 1$ giving (\refeq{tointerpolate2}). The case of dimension $n\neq 4$ differs only in the arithmetic choices of exponents in H\"older's inequality.

  \end{proof}

\section{Generalized Seiberg-Witten Equations}
\label{section4}

This section introduces generalized Seiberg-Witten equations; the subsequent section shows that in the relevant cases these fit into the framework of Sections \ref{section2}--\ref{section3}. 

\subsection{Seiberg-Witten Data}
\label{section4.1}
Generalized Seiberg-Witten equations are systems of coupled non-linear first-order PDEs on manifolds of dimension 3 and 4. There is a system of generalized Seiberg-Witten equations associated to each quaternionic representation of a compact Lie group $G$. Rather than working in the most general setting, we will here opt for an abridged exposition which suffices for our purposes. For a more general introduction, see
 \cite{WalpuskiZhangCompactness,WalpuskiNotes,Doanthesis}. 
 
We first focus on the three-dimensional case. Let $(Y,g_0)$ denote an oriented Riemannian 3-manifold, and fix a spin structure $\frak s\to Y$ considered as a principal $\text{Sp}(1)\simeq \text{Spin}(3)$-bundle. We may write \be T^\star Y\simeq \frak s\times_{Ad}\text{Im}(\mathbb H)\label{cotangent}\ee 
 where $Ad: \text{Sp}(1)\to \text{Im}(\mathbb H)=\frak{sp}(1)$ is the adjoint representation of $\text{Spin}(3)$. Next, let $G$ be a compact Lie group, and $V$ a quaternionic vector space with a real inner-product denoted $\br -,-\kt$ carrying a quaternionic representation $$\rho: G\to \text{GL}_{\mathbb H}(V)$$ \noindent respecting the inner product. 
 
 Since $\rho$ is quaternionic, it extends to a map (denoted by the same symbol)  $\rho:\text{Sp}(1)\times G\to \text{GL}_{\mathbb H}(V)$ given by $(q,g)\mapsto q\cdot \rho(g)$. Let  \be\gamma: \text{Im}(\mathbb H)\otimes_\R \frak g \to \text{End}_{\mathbb H}(V).\label{gamma}\ee
 
 \noindent be the induced linearized action. Additionally, for $\Psi\in V$ define $\mu: V\to \text{Im}(\mathbb H)\otimes_\R \frak g$ by 
 
\be \frac{1}{2}\mu(\Psi,\Psi)=\frac{1}{2}\sum_{\alpha,j}\br \gamma(I_j\otimes \frak t_\alpha)\Psi, \Psi\kt  \ I_j\otimes t_\alpha \label{mu}\ee
 \noindent where $I_j$ for $j=1,2,3$  is a basis of $\text{Im}(\mathbb H)$, and $\{\frak t_\alpha\}$ are a basis of $\frak g$. It is straightforward to check that $\frac{1}{2}\mu$ is the {\bf hyperk\"ahler moment map} for the action of $G$ on $V$ by $\rho$. 
 
 Working globally on $Y$ now, let $P\to Y$ be a principal $G$-bundle. 

\medskip 
\begin{defn}The {\bf Spinor Bundle} associated to $\rho$ is the vector bundle
$$S= (\frak s\times_Y P) \times_{\text{Sp}(1)\times G} V.$$

\noindent It is endowed with a { Clifford multiplication} and a moment map denoted (respectively) by $$\gamma: \Omega^1(\frak g_P) \to \text{End}(S) \hspace{2cm} \mu: S \to \Omega^1(\frak g_P)$$
\noindent given fiberwise by the maps (\refeq{gamma}) and (\refeq{mu}). Here, $\Omega^1(\frak g_P)$ is the space of 1-forms valued in the adjoint bundle of $P$, viewed as the associated bundle via (\refeq{cotangent}) and the adjoint representation of $G$. 
\end{defn} 

Each set of data $(G, V,\rho)$ gives rise to a system of generalized Seiberg-Witten equations on $(Y,g_0)$. These are the following PDEs for a pair $(\Psi, A) \in \Gamma(S)\times \mathscr A(P)$ consisting of a spinor  $\Psi$ and a connection $A$ on the principal bundle $P$ respectively. Such a pair is called a {configuration}. 

\begin{defn}
The {\bf generalized Seiberg-Witten equations} determined by $(P, S, \gamma,\mu)$ are 
\begin{eqnarray}
\slashed D_A \Psi&=& 0\label{SW1} \\
 \star F_A  + \tfrac{1}{2}\mu(\Psi,\Psi)&=&0 \label{SW2}
\end{eqnarray}
where $\slashed D_A$ denotes the Dirac operator on $S$ determined by the spin connection and $A$ using the Clifford multiplication $\gamma$, and $F_A$ is the curvature of $A$. These equations are invariant under the {Gauge group} $$\mathcal G= \Gamma(P\times_{Ad} G).$$
\end{defn}

The discussion above carries over to the case of an oriented Riemannian 4-manifold $(X,g_0)$ with only two minor modifications.

\begin{enumerate}
\item[(1)]As not all 4-manifolds are spin, we impose the additional requirement that there exists a central element $-1\in Z(G)$ so that $\rho(-1)=-\text{Id}$, and consider principal bundles $P,Q$ with structure groups $G$ and $\text{Spin}^G(4)= (\text{Spin}(4)\times  G)/\Z_2$ respectively. Here, $\Z_2$ acts by $(1,1)\mapsto (-1,-1)$.
\item[(2)] In (\refeq{cotangent}), $\Lambda^1(Y)$ is replaced by $\Lambda^2_+(X)= Q\times_{\sigma} \text{Im}(\mathbb H)$ where $\sigma:\text{Spin}^G(4)\to SO(\text{Im}(\mathbb H))$ is the composition of projection to $SO(4)$ via $\text{Spin}(4)$ and the standard 3-dimensional representation of $SO(4)$. Similarly in (\refeq{SW1}--\refeq{SW2}), $\star F_A$ is replaced by $F_A^+$ and $\slashed D_A$ by $\slashed D_A^+$.   
\end{enumerate}

\noindent See  \cite{WalpuskiNotes} for a more complete discussion of the 4-dimensional case.

\bigskip 

Most equation of interest in mathematical gauge theory arise as generalized Seiberg-Witten equations for particular choices of the data $(G,V,\rho)$.

\begin{eg} The Anti-Self-Dual (ASD) Yang-Mills equations and standard Seiberg-Witten equations (and their dimensional reductions) are generalized Seiberg-Witten equations obtained from the following data. 

\begin{enumerate}
\item[$\bullet$] The data $G=SU(2)$ and $V=\{0\}$ with $\rho$ being the trivial representation gives the ASD Yang-Mills equations. 

\item[$\bullet$] The data $G=U(1)$ and $V=\mathbb H$ with $\rho$ being multiplication by $e^{i\theta}\in U(1)$ on the right reproduces the standard Seiberg-Witten equations. 
\end{enumerate}
\end{eg}

\medskip

\bigskip 
We now distinguish four cases, {\bf (I)--(IV)} as in the statement of Theorem \ref{mainb}. These cases will be referred to repeatedly throughout the proof of Theorem \ref{mainb}.

\begin{eg}[{\bf {Case (I)}}] \label{Swmultispinor} The data $G=U(1)$ and $V=\mathbb H \otimes_\C \C^r$ with $\rho$ being multiplication by $e^{i\theta}\in U(1)$ on the right on the first factor leads to the Seiberg-Witten equations with $r$ spinors. \newline \  \indent \hspace{.5cm}    More concretely, in this case $S= W \otimes E$ where $W\to Y$ is the spinor bundle of a $\text{spin}^c$ structure and $E\to Y$ an auxiliary bundle of (complex) rank $r$ with trivial determinant. Then the equations are 
\bea
\slashed D_{A}\Psi&=& 0 \\
\star F_A +\tfrac{1}{2}\mu(\Psi,\Psi)&=& 0 
\eea

\noindent where $\slashed D_A$ is the Dirac operator on $S_V$ and the moment map $\mu$ can be described as follows. Let $e^j$ denote a local frame of $T^\star Y$; the local frame of $W$ arising as the $\pm i$ eigenspaces of Clifford multiplication by $e^1$, a spinor can be written $\Psi=(\alpha,\beta)$ where $\alpha,\beta \in \Gamma(V)$. Then, the moment map is,  \be \frac{1}{2}\mu(\Psi,\Psi) = \frac{1}{2}\sum_{j=1}^3 \br  i e^j. \Psi, \Psi\kt  ie^j=  \frac{i}{2}\Big ((|\beta|^2 -|\alpha|^2) e^1 \ , \  \text{Re}(-\overline \alpha \beta) e^2 \ , \  \text{Im}(-\overline \alpha \beta)e^3\Big ).\label{2spinormomentmap} \ee

\noindent In particular $\mu$ is the sum of the standard Seiberg-Witten moment map over the $r$-spinors. See \cite{PartI, DWExistence} for further details. 

\end{eg}

\bigskip 

\begin{eg}[{\bf {Case (III)}}] \label{Swmultispinor} On a 4-manifold $X$ the equations and moment map are analogously related to the standard 4-dimensional Seiberg-Witten equations. To be precise, they take the form 
\bea
\slashed D^+_{A}\Psi&=& 0 \\
 F_A ^+ +\tfrac{1}{2}\mu(\Psi,\Psi)&=& 0 
\eea
\noindent where $\mu$ is defined analogously to (\refeq{2spinormomentmap}) using an orthonormal basis $\omega^j$ of $\Lambda^{2}_+(i\R)$ in place of the basis $e^j$.  

\end{eg}

\bigskip

\begin{eg} ({\bf {Case (II)}}).\label{case2eg}
The data $G=SU(2)$ and $V=\frak{su}(2)\otimes_\R \mathbb H$ with $\rho$ being the quaternionificaition of the adjoint representation gives rise to the equations for a flat $\text{SL}(2,\C)$ connection in dimension $3$. In this case, $$S= (\Lambda^0\oplus \Lambda^1)(\frak g_P) \hspace{2cm} \slashed D_A= \bold d_A= \begin{pmatrix} 0 & -d_A^\star \\ -d_A & \star d_A \end{pmatrix}\begin{pmatrix}\Psi_0 \\ \Psi_1\end{pmatrix}$$ 

\noindent and the moment map is $\frac{1}{2}\mu(\Psi,\Psi)=- \frac{1}{2}\star [\Psi\wedge \Psi].$ The equations therefore become 
\bea
\bold d_A \Psi &=& 0 \\
\star F_A -\tfrac{1}{2}\star [\Psi \wedge \Psi]&=&0 
\eea

\end{eg}

\bigskip 
\begin{eg} ({\bf {Case (IV)}}). In dimension 4, the data of Example \ref{case2eg} gives rise to the complex ASD equations. Analogously to the 3-dimensional case, in dimension $4$ one has
 $$ S^+ = \Lambda^1(\frak g_P) \hspace{1cm} S^-=(\Lambda^0\oplus \Lambda^2_-)(\frak g_P) \hspace{2cm}\bold d_A^+=(d_A^\star, d_A^-)$$
 \noindent and $\frac{1}{2}\mu(\Psi,\Psi)= -\frac{1}{2}[\Psi \wedge \Psi]^+.$ The corresponding equations are 
 \bea
\bold d^+_{A}\Psi&=& 0 \\
 F_A ^+ -\tfrac{1}{2}[\Psi \wedge \Psi]^+&=& 0. 
\eea
\end{eg}
\bigskip 

\begin{eg} {\bf (Vafa-Witten Equations)}\label{VWeg}
The same data as in the previous example can also give rise to the Vafa-Witten equations on  a four-manifold $X^4$ (depending on a choice of auxiliary data omitted from the discussion here---see \cite{WalpuskiNotes}, Examples 2.31 and 2.36). In this case, one has 

$$ S^+ = (\Lambda^0 \oplus \Lambda^2_+)(\frak g_P) \hspace{1cm} S^-=\Lambda^1(\frak g_P) \hspace{2cm} \slashed D_A(C,B)= d_AC + d_A^\star B $$
\noindent and if $\Psi= (C,B) \in (\Omega^0 \oplus \Omega^2_+)(\frak g_P)$ then the moment map is  $$\frac{1}{2}\mu(\Psi, \Psi)=[C,B] + \frac{1}{2}[B\times B]$$
\noindent where $[ \_ \times \_]$ is the product induced by viewing $\Lambda^2_+(TX)$ as the bundle of Lie algebras arising as the associated bundle of $SO(4)$-frames on $X$ via the positive irreducible component of the adjoint representation on $\frak{so}(4)$. Explicitly, in an orthonormal frame $\omega_i$ of $\Lambda^2_+$ it is given on a self-dual 2-form $B= \omega_i\otimes B_i$ by $[B\times B]=\epsilon_{ijk} [B_i, B_j]\omega_k$ where summation over repeated indices is implicit. 

\end{eg}
\bigskip 

\begin{eg}{\bf ($\text{ADHM}_{r,k}$ Seiberg-Witten Equations)}\label{ADHMeg}
 For $G=U(k)$ and $$V_{r,k}=\text{Hom}_\C(\C^r, \mathbb H\otimes_\C \C^k) \oplus (\mathbb H^\vee \otimes_\R \frak u(k))$$ where $\rho$ acts on the $\C^k$ factor via the standard representation and the $\frak u(k)$ factor via the adjoint gives rise to the $\text{ADHM}_{r,k}$ Seiberg-Witten Equations. Here $\mathbb H^\vee$ denotes the dual space. The zero-locus $\mu^{-1}(0)$ of the moment map in this situation coincides with the ADHM construction of the moduli space of  ASD $SU(r)$-instantons of charge $k$ on $\R^4$. These equations are conjectured to connect  Yang-Mills theory on manifolds with special holonomy to Seiberg-Witten theory on calibrated submanifolds (see \cite{DonaldsonSegal, HaydysG2SW, DWAssociatives}).  
 
 When $k=1$ these equations coincide with Example \ref{Swmultispinor}. For $(r,k)=(1,2)$ the spinor bundle is (effectively) the sum of those in Case (I) and Case (II), with the moment map being the sum of the ($U(2)$-analogue) of the moment maps for those. This particular case is studied in detail in  \cite{WalpuskiZhangCompactness}
\end{eg}

\subsection{$\Z_2$-Harmonic Spinors and Compactness} 
\label{section4.2}
This subsection summarizes known compactness results for the moduli space of solutions to generalized Seiberg-Witten equations; these were established in \cite{Taubes3dSL2C, Taubes4dSL2C,HWCompactness,TaubesU1SW,WalpuskiZhangCompactness,TaubesVW}. See also 
\cite{PartI, DWDeformations, HWCompactness,WalpuskiZhangCompactness} for additional details and exposition.  It is well-known that the moduli space of solutions to the standard Seiberg-Witten equations is compact \cite{MorganSW, KM}. This is a consequence of the inequality 

$$ \br \gamma(\mu(\Psi,\Psi))\Psi, \Psi \kt\geq  \frac{1}{4}|\Psi|^4$$

\noindent which leads to an {\it a priori} bound on the $L^2$-norm of the spinor for a solution of (\refeq{SW1}--\refeq{SW2}). For general Seiberg-Witten data, however, no such inequality can hold because $\mu^{-1}(0)\neq \emptyset$ in general. Consequently, there may be sequences of solutions $(\Psi_i, A_i)$ with $\|\Psi_i\|_{L^2}\to \infty$, which have no convergent subsequences and therefore lead to a loss of compactness of the moduli space.

In this situation, one may attempt to compactify the moduli space by ``blowing up'' the configuration space, i.e. extending it to include a boundary stratum at infinity. More specifically, we re-parameterize the subset of the configuration space with $\|\Psi\|_{L^2}>0$ by replacing $\Psi$ by a pair $(\Phi, \e)$ where $\|\Psi\|_{L^2}=\frac{1}{\e}$ and $\Phi= \e\Psi$ is a spinor with unit $L^2$-norm. After including the boundary stratum consisting of configurations with $\e=0$, the blown-up configuration space consists of triples $(\Phi,A,\e) \in \mathbb S(\Gamma(S))\times \mathscr A(P)\times [0,\infty]$, where $\mathbb S(\Gamma(S))$ denotes the sphere of spinors with unit $L^2$-norm.  The corresponding {\bf blown-up Seiberg-Witten equation} is 
\begin{eqnarray}
\slashed D_A \Phi&=& 0\label{bSW1} \\
  \star \e^2 F_A  + \tfrac{1}{2}\mu(\Phi,\Phi)&=&0  \label{bSW2} \\ 
 \|\Phi\|_{L^2}&=&1. \label{bSW3}
\end{eqnarray}

Intuitively, one expects that a sequence of solutions to the original equations with diverging $L^2$-norm should converge to a solution of the $\e=0$ verison of the blown-up equations, and that including these solutions as boundary strata would result in a compact moduli space. The upcoming Theorem \ref{compactness} establishes a version of this statement, but there are several important caveats: 

\begin{enumerate}
\item[1)] the $\e=0$ version of (\refeq{bSW2}) demands that $\Phi$ lies in the set $\mu^{-1}(0)$ in each fiber; the latter is a closed subset of each fiber of $S$, but is not a manifold because $0$ is a cone point in each fiber. 
\item[2)] The energy density $|F_A|^2$ of the curvature may concentrate along subsets of $Y$ (resp. $X$) in the limit $\e\to 0$. 

\end{enumerate}

\noindent Because of these complications, the limit of a sequence of solutions only satisfies the $\e=0$ version of (\refeq{bSW2}) away from a closed subset denoted $\mathcal Z$ called the {\bf singular set}. In the case of a non-abelian gauge group on a 4-manifold, the bubbling locus arising from Uhlenbeck compactness is also included in $\mathcal Z$. 

The following theorem combines compactness results for several generalized Seiberg-Witten equations which were proved independently by multiple authors. It unifies results on Case (I) the Seiberg-Witten equations with $r=2$ spinors in 3 dimensions \cite{HWCompactness}, Case (II) the equations for a flat $\text{SL}(2,\C)$-connection on a 3-manifold  \cite{Taubes3dSL2C,WalpuskiZhangCompactness}, Case (III) the Seiberg-Witten equations with $r=2$ spinors in 4 dimensions \cite{TaubesU1SW}, and Case (IV) the complex ASD equation in 4 dimensions \cite{Taubes4dSL2C}. In addition, it includes regularity statements for the singular set proved in \cite{TaubesZeroLoci, ZhangRectifiability}. In each case, if a sequence of solutions has a subsequence on which the $L^2$-norm remains bounded (i.e. if $\limsup \e>0$) then standard compactness arguments apply to show a subsequence converges; thus we state the theorem only in the case where $\e\to 0$.  

\begin{thm}\label{compactness} (\cite{HWCompactness,Taubes3dSL2C,TaubesU1SW, Taubes4dSL2C, TaubesZeroLoci,ZhangRectifiability})
Suppose that $Y$ is a closed, oriented 3-manifold (respectively, $X$ a 4-manifold) and $(P,G,\rho,\mu)$ generalized Seiberg-Witten data corresponding to Cases (I), (II) from the statement of Theorem \ref{mainb} (resp. Cases  (III), (IV)). Given a sequence $(\Phi_i, A_i, \e_i)$ of blown-up configurations satisfying (\refeq{bSW1}--\refeq{bSW3}), i.e. $$\slashed D_{A_i}\Phi_i =0 \hspace{1.5cm}\star  \e_i^2 F_{A_i}+\frac{1}{2}\mu(\Phi_i,\Phi_i)=0 \hspace{1.5cm} \|\Phi_i\|_{L^2}=1  $$
  \noindent (resp. $\slashed D_{A_i}^+$ and  $F_{A_i}^+$) with respect to a sequence of metrics $g_i \to g_0$ on $Y$ (resp. $X$), such that $\e_i\to 0$.

 \noindent Then, there exists a triple $(\mathcal Z_0,\Phi_0, A_0)$ where  
\begin{itemize}
\item $\mathcal Z_0 \subseteq Y$ (resp. $X$) is a closed rectifiable subset of Haudorff codimension at least 2. 
\item $\Phi_0$ is a spinor on $Y-\mathcal Z_0$ such that $|\Phi_0|$ extends as a continuous function to $Y$ (resp. $X$) with $\mathcal Z_0=|\Phi_0|^{-1}(0)$. 
\item $A_0$ is a connection on $P|_{Y-\mathcal Z_0}$ (resp. $P|_{X-\mathcal Z_0})$, 
\end{itemize}
\noindent such that $(\Phi_0, A_0)$ satisfies the $\e=0$ version of (\refeq{bSW1}--\refeq{bSW3}) on $Y-\mathcal Z_0$ (resp. $X-\mathcal Z_0$) with respect to the metric $g_0$. Furthermore, there is an $\alpha>0$ such that and after passing to a subsequence and up to gauge transformations defined on $Y-\mathcal Z_0$ (resp. $X-\mathcal Z_0$), 
 
 \be
 \Phi_i \overset{L^{2,2}_{loc}}\rightharpoonup \Phi_0 \hspace{1.5cm} A_i \overset{L^{1,2}_{loc}}\rightharpoonup A_0 \hspace{1.5cm} |\Phi_i| \overset{C^{0,\alpha}}\to |\Phi_0|.  \label{convergencetolim}
 \ee 
Here, local convergence means on compact subsets of $Y-\mathcal Z_0$ (resp. $X$), and the half-arrows in the first two statements denote convergence in the weak topology. In Case (III) the convergence  $\Phi_i\to \Phi_0$ is (strongly) $L^{2,2}_{loc}$ and in Case (IV) it is only (strongly) $L^{1,2}_{loc}$. 

Finally, in Case (II) and Case (IV) the limiting connection $A_0$ has harmonic curvature, i.e. it satisfies $d_{A_0}F_{A_0}= d_{A_0}^\star F_{A_0}=0$ where $d_{A_0}$ is the exterior covariant derivative.
\qed
\end{thm}

\bigskip
 
The limiting configuration $(\mathcal Z_0, \Phi_0, A_0)$ satisfies the $\e=0$ version of (\refeq{bSW1}--\refeq{bSW3}), which reads 

\be \slashed D_{A_0}\Phi_0 =0 \hspace{1.5cm} \mu(\Phi_0,\Phi_0)=0 \hspace{1.5cm} \|\Phi_0\|_{L^2}=1,  \label{e=0version}\ee
and is considered up to the action of gauge transformations on $Y-\mathcal Z_0$ (resp. $X-\mathcal Z_0$). This equation is not elliptic, even modulo gauge, as the symbol of ${\bold d_A}$ degenerates in the limit $\e\to0$. The {\bf Haydys Correspondence} (see \cite{PartI} Section 2, \cite{DWDeformations}), however, shows that by exploiting gauge invariance in a different way than gauge-fixing, (\refeq{e=0version}) can be recast as an elliptic equation whose symbol degenerates only along $\mathcal Z_0$.

The Haydys Correspondence may be paraphrased as follows (see \cite{DWDeformations} Section 4 and  \cite{HaydysCorrespondence} for details). Let $\pi: \mu^{-1}(0)\to \mathfrak X=\mu^{-1}(0)/G$ be the projection map to the fiberwise quotient by the action of $G$; $\mathfrak X$ is a bundle whose fibers are hyperk\"ahler orbifolds isometric to the hyperk\"ahler quotient $V///G$.  

Given a solution $(\mathcal Z_0,\Phi_0, A_0)$ of (\refeq{e=0version}), the projection $s=\pi(\Phi_0)\in \Gamma(\mathfrak X)$ is a solution of a different PDE, the {\it Fueter equation} $\mathfrak F(s)=0$. The content of the Haydys correspondence is that one can recover the triple $(\mathcal Z_0,\Phi_0, A_0)$ from $s$ even though $s$ retains no information about the connection $A_0$ or the action of $G$. This is done by choosing a local lift $\Phi_0$ of $s$ to $\mu^{-1}(0)\subset \Gamma(S)$ on $Y-\mathcal Z_0$ (resp. $X-\mathcal Z_0$). Such a lift determines a splitting  of $S$ along the image of $\Phi$ as follows. First, decompose $S= T(\mu^{-1}(0))\oplus T(\mu^{-1}(0))^\perp$; the first factor further decomposes as $T(\mu^{-1}(0))= S^\text{Re}_{\Phi_0} \oplus \frak g\Phi_0$ \ \footnote{The notation of the superscripts is chosen to agree with the case of the Seiberg-Witten equations with $r=2$ spinors where the spinor bundle admits a real structure $\tau$ with $\tau^2=Id$, see Sections 2--3 of \cite{PartI}. Other authors \cite{DWDeformations} denote these $\mathfrak H_{\Phi_0}, \mathfrak N_{\Phi_0}$, while we reserve the latter notation for the splitting in (\refeq{fixeddegeneracy}) } , where $\frak g \Phi_0=\{ v \Phi_0 \ | \ v\in \mathfrak g_P\}$. The decomposition determined by $\Phi_0$ is then \be S|_{Y-\mathcal Z_0}= S^\text{Re}_{\Phi_0}\oplus S^\text{Im}_{\Phi_0}\label{ReImSplitting}\ee
\noindent where $S^\text{Im}_{\Phi_0}=\frak g\Phi_0 \oplus T(\mu^{-1}(0))^\perp$. It can then be shown that the condition that $\nabla_{A_0}\Phi_0 \in \Gamma(S^\text{Re}_{\Phi_0})$ determines $A_0$ uniquely, in which case $\frak F(s)=0$ is equivalent to (\refeq{e=0version}). Notice that one either side of the Haydys correspondence the singular behavior along $\mathcal Z_0$ cannot be eliminated. 
\bigskip 

In Cases (I)--(IV) of Theorem \ref{mainb} (and Examples \ref{VWeg}, \ref{ADHMeg}),  the Haydys Correspondence and the Fueter equation admit a simplification due to the following additional structure: there exists a linear subspace $E\subset \mu^{-1}(0)\subset V$ such that every $G$-orbit intersects $E$ in exactly 2 points. In this situation, for a local lift valued in $E$, one has $S^\text{Re}_{\Phi_0}=E$ and the Fueter equation is simply the Dirac equation on spinors considered up to sign. The data of a solution of (\refeq{e=0version}) is then equivalent to the following data. 

\begin{defn}\label{Z2Harmonic}
A {\bf $\Z_2$-Harmonic Spinor} valued in a (real) Clifford module $E\to Y$ (resp. $X$) is a triple $(\mathcal Z_0, \ell_0, \Phi_0),$ where 
\begin{enumerate}
\item[(1)] $\mathcal Z_0 \subset Y$ (resp. $X$) is a closed, rectifiable subset of Hausdorff codimension 2,
\item[(2)] $\ell_0 \to Y- \mathcal Z_0$  (resp. $X-\mathcal Z_0$) is real line bundle, and  
\item[(3)] $\Phi_0 \in \Gamma(E \otimes_\R \ell_0 )$ is a spinor with $\nabla \Phi_0\in L^2$ and whose norm $|\Phi_0|$ extends to $Y$ (resp. $X$) as a H\"older continuous function, 
\end{enumerate} such that  
\be \|\Phi_0\|_{L^2}=1 \hspace{1.5cm} \slashed D_{\mathcal Z_0}\Phi_0=0 \hspace{1.5cm} |\Phi_0|^{-1}(0)=\mathcal Z_0.\label{Z2Dirac}\ee  

\noindent $\Z_2$-harmonic spinors are always considered up to the equivalence $\Phi_0 \mapsto - \Phi_0$. 
\end{defn}    

\bigskip

In (\refeq{Z2Dirac}), the Dirac operator is formed using the connection arising from the connection on $E$ and the unique flat connection with holonomy in $\Z_2$ on the line bundle $\ell_0$. Under the Haydys correspondence, this unique flat connection of $\ell_0$ is equivalent to the connection arising from $A_0$ in the conclusion of Theorem \ref{compactness} (see \cite{PartI} Section 3) and the spin connection. For all the cases of Theorem \ref{mainb}, the $\Z_2$-harmonic spinors that arise are sections of Clifford modules of real rank 4. In Case (II), and Case (IV) the $\Z_2$-harmonic spinors arising from Theorem \ref{compactness} are also called $\Z_2$-harmonic 1-forms, i.e. spinors for the Clifford modules $(\Omega^0\oplus \Omega^1)(\R)$ or $\Omega^1(\R)$ in 3 and 4 dimensions respectively.

\section{Concentration Properties of Generalized Seiberg-Witten Equations}
\label{newsection5}

As explained in the introduction, it is desirable to improve the convergence statements in Theorem \ref{compactness} to $C^\infty_{loc}$. Although the equations are elliptic for $\e\neq 0$, naive attempts to bootstrap convergence are foiled by increasingly large powers of $\e$ entering the elliptic estimates (see Section \ref{section5}). The abstract framework introduced in Section \ref{section2}--\ref{section3} can be applied to overcome this problem. Let $(\Phi_i, A_i, \e_i)$ be a sequence of solutions to (\refeq{bSW1}--\refeq{bSW3}) converging to a $\Z_2$-harmonic spinor $(\mathcal Z_0, A_0,\Phi_0)$ in the sense of Theorem \ref{compactness}. Then the un-renormalized configuration may be written as a perturbation of the limit $(\frac{\Phi_i}{\e_i}, A_i)=(\frac{\Phi_0}{\e_i}, A_0)+ (\ph_i, a_i)$, where $(\ph_i, a_i)$ solve the equation

\be \mathcal L_{(\frac{\Phi_0}{\e_i}, A_0)} (\ph_i, a_i) + Q(\ph_i, a_i)=-E_0 \label{deformationeq}\ee

\noindent Here,  $\mathcal L_{(\Psi,A)}$ denotes the linearized Seiberg-Witten equations linearized at the configuration $(\Psi,A)$, $Q(-,-)$ is a quadratic term, and $E_0=SW(\Phi_0, A_0)$ is the error by which the limiting configuration fails to solve the Seiberg-Witten equations. The next two subsections show that (\refeq{deformationeq}) behaves as a concentrating Dirac operator with fixed degeneracy as $\e\to 0$, with the singular set $\mathcal Z_0$ occupying the role of the set denoted by the same symbol in Sections \ref{section2}--\ref{section3}.

\subsection{The Concentration Property}
In this subsection, it is shown that the linearization of the generalized Seiberg-Witten equations associated to any data satisfy the concentration property of Definition \ref{concentrationproperty}. In Subsection \ref{section4.4}, it is shown that the non-linearity of the equations satisfies the similar criteria necessary for Corollary \ref{cornonlinear} to apply. 

On a 3-manifold, the linearization of the generalized Seiberg-Witten equations (\refeq{SW1}--\refeq{SW2}) is as follows. Let $(\frac{\Phi}{\e}, A)$ denote a configuration, where  $\|\Phi\|_{L^2}=1$ and $\e>0$. The linearization at $(\frac{\Phi}{\e}, A)$ acting on a variation $(\ph,a)$ is 

$$\d{}{s}\Big |_{s=0} SW\left(\tfrac{\Phi}{\e}+ s\ph, A+ sA\right)=\begin{pmatrix} \slashed D_A\ph + \gamma(a)\tfrac{\Phi}{\e} \vspace{.15cm}\\ \tfrac{\mu(\ph, \Phi)}{\e}+ \star d_A a\end{pmatrix}.$$

\noindent where $\mu( \ph, \psi)$ denotes the polarization of the moment map. To make this into an elliptic system, we supplement the configuration $(\frac{\Phi}{\e},A)$ with an auxiliary $0$-form $a_0\in \Omega^0(\frak g_P)$ and impose the gauge-fixing condition 
\be -d_{A}^\star a + \frac{\mu_0(\ph, \Phi)}{\e}=0 \label{gaugefixing}\ee

\noindent where we have updated our notation so that $a=(a_0, a_1)\in (\Omega^0\oplus \Omega^1)(\frak g_P)$ and $\mu=(\mu_0, \mu_1) \in (\Omega^0\oplus \Omega^1)(\frak g_P)$. Here, $\mu_0: S\otimes S\to \frak g_P$ is defined in a local trivialization $t_\alpha$ of $\frak g_P$ by $\mu_0(\ph,\psi):= \sum_\alpha \br  \frak t_\alpha \ph, \psi\kt \frak t_\alpha$. The equations may then be written in the suggestive form:

\be \mathcal L_{(\Phi, A,\e)}\begin{pmatrix} \ph \\ a \end{pmatrix}=\left(D+\frac{1}{\e}\mathcal A\right)\begin{pmatrix} \ph \\ a \end{pmatrix}\label{linearizationasCDO}\ee   

\noindent where 
\be D=\begin{pmatrix}\slashed D_A &0 \vspace{.15cm}\\ 0 & \bold d_A \end{pmatrix}  \ \  \hspace{3cm}\ \  \mathcal A= \begin{pmatrix}0 & \gamma(\_)\Phi \vspace{.15cm}\\ \mu(\_,\Phi) & 0 \end{pmatrix}, \label{4.7} \ee

\noindent and $\bold d_A= \begin{pmatrix} 0 & -d_A^\star \\ -d_A & \star d_A \end{pmatrix}\begin{pmatrix}a_0 \\ a_1\end{pmatrix}.$   On a 4-manifold $X$, the auxiliary form $a_0$ is not necessary for ellipticity, and the formula (\refeq{linearizationasCDO}) for $\mathcal L_{(\Phi,A,\e)}$  is the same after replacing $\slashed D_A$ by  $\slashed D_A^+$ and $\bold d_A$ by $\bold d_A^+=(-d_A^\star, d_A^+)$. 


\medskip 

The symbol of $D$ in \refeq{4.7} is $$\sigma_D(\xi)= \begin{pmatrix}  \rho(\xi)& 0  \vspace{.15cm}\\ 0 & \bold{cl}(\xi)\end{pmatrix}\hspace{1cm} \text{for }\hspace{1cm}\xi \in \Omega^1(T^\star Y). $$
\noindent where $\rho, {\bf cl}$ are the symbols of $\slashed D_A, \bold d_A$ respectively (or $\slashed D_A^+, \bold d_A^+$ in dimension 4). The next lemma is a particular instance of  a more general result concerning commuting Clifford pairs discussed in \cite{ManosConcentrationI, ManosConcentrationII}. 

\begin{lm}\label{concentrationL}
The linearization of the generalized Seiberg-Witten equations written in the form (\refeq{linearizationasCDO})--(\refeq{4.7}) obeys the Concentration Property of Definition \ref{concentrationproperty}, i.e. $\sigma_D$ and $\mathcal A$ satisfy 

\be \mathcal A^\star  \sigma_D(\xi) = \sigma_D(\xi)^\star  \mathcal A\label{concentrationLstatement}\ee \noindent for all $\xi \in T^\star Y$.  \end{lm}

 \begin{proof}
 For each $\xi \in \Omega^1(\R)$, (\refeq{concentrationLstatement}) is equivalent to the following equalities for all $a\in (\Omega^0\oplus \Omega^1)(\frak g_P)$,  and $\ph, \psi \in \Gamma(S)$:  
\begin{eqnarray}
\rho(\xi)\gamma(a)&=&-\gamma({\bf cl}(\xi)a)\label{usefulid1} \\
{\bf cl}(\xi) \mu(\ph,\psi)&=& -\mu(\rho(\xi)\ph, \psi). \label{usefulid2}
\end{eqnarray}
In (\refeq{usefulid1}), ${\bf cl}$ is extended to act on $\frak g_P$-valued forms via the form components.  The expressions (\refeq{usefulid1}--\refeq{usefulid2}) are  easily verified in an oriented orthonormal frame of $T^\star Y$ (resp. $T^\star X$) and $\frak g_P$.

 \end{proof}

\subsection{Fixed Degeneracies}
\label{section4.4}

This subsection shows that the  fixed degeneracy assumption of Definition \ref{concentrationproperty} is satisfied for the linearized Seiberg-Witten equations, and that the non-linear terms have the form necessary for Corollary \ref{cornonlinear} to apply. 

To begin, there is the following alternative description of the subspaces $S^\text{Re}, S^\text{Im}$ from (\refeq{ReImSplitting}). Since $\frac{1}{2}\mu(\Phi, \Phi)$ is bilinear, its linearization at a spinor $\Phi_0$ is given by $\mu(-,\Phi_0)$, where $\mu$ is now interpretted as the bilinear form arising as the polarization of the original quadratic map. On $Y-\mathcal Z_0$ where $\Phi_0$ is non-vanishing, one then has , \be S^\text{Re}_{}= \ker(\mu(-,\Phi_0)) \hspace{1.3cm}\text{}\hspace{1.3cm}S_{}^\text{Im}=(S_{}^\text{Re})^\perp \label{SReDef}\ee  

\noindent  (see \cite{Doanthesis}, Proposition 2.1.5  and \cite{PartI}, Section 2). The next proposition is proved on a case by case basis for the different equations to which Theorem \ref{mainb} applies. The result could be shown using a more abstract framework, but  we find it instructive to give explicit descriptions as the splitting described by the proposition provides a novel way of writing many of these equations which may be useful elsewhere. 
\begin{prop}\label{fixeddegeneracyL}
The linearization of the generalized Seiberg-Witten equations written in the form (\refeq{linearizationasCDO})--(\refeq{4.7}) obeys the fixed degeneracy assumption of Definition \ref{concentrationproperty}, i.e. in dimension $n=3$ there is a splitting of vector bundles  $$S\oplus (\Omega^0\oplus \Omega^1)(\frak g_P)=\frak N\oplus \frak H$$   
 which respects Clifford multiplication $\gamma$ and is parallel with respect to $\nabla_{A_0}$ such that the map $\mathcal A$ of (\refeq{4.7}) has the block diagonal form (\refeq{fixeddegeneracy}). 
 
 In dimension $n=4$ there are splittings $$S^+\oplus \Omega^1(\frak g_P) = \frak N^+ \oplus \frak H^+ \hspace{1cm}\text{ and }\hspace{1cm}S^-\oplus (\Omega^0\oplus \Omega^2_+) (\frak g_P)= \frak N^- \oplus \frak H^-$$ for which the same conclusions hold. 
\end{prop}

\begin{proof} The proposition is proved separately for cases (I)--(VI) of Theorem \ref{mainb}. 
\medskip

\noindent{\bf Case (I):} (Two spinor Seiberg-Witten on $Y^3$).  In this case, the $\Z_2$-harmonic spinors are as in Definition \ref{Z2Harmonic} with $E$ the spinor bundle of a spin structure on $Y$. A limiting configuration $(\mathcal Z_0, A_0, \Phi_0)$ satisfying (\refeq{e=0version}) gives rise to such a $\Z_2$-harmonic spinor as follows: the Haydys Correspondence gives an isomorphism $S^\text{Re}_{\Phi_0}\simeq E\otimes \ell_0$, and under this isomorphism the connection induced on $S^\text{Re}$ by $A_0$ is intertwined with the connection formed from the spin connection on $E$ and the unique flat connection on $\ell_0$ with holonomy in $\Z_2$. Thus in this case, the limiting connection $A_0$ on $Y-\mathcal Z_0$ is itself flat with holonomy in $\Z_2$ (see also \cite{HWCompactness} Appendix I and Sections 2--3 of \cite{PartI}). We tacitly also refer to the triple $(\mathcal Z_0, A_0, \Phi_0)$ as a $\Z_2$-harmonic spinor.

Let  $S^\text{Re}\to Y-\mathcal Z_0$ be the bundle defined in (\refeq{SReDef}). Then define \be \frak N:=S^\text{Re}\hspace{3cm} \frak H:=S^\text{Im} \oplus (\Omega^0\oplus \Omega^1)(i\R).\label{parallelsplittingCase1}\ee

\noindent The splitting \refeq{parallelsplittingCase1} is respected by Clifford multiplication by forms in $\alpha \in \Omega^0(\R)\oplus \Omega^1(\R)$. Indeed, the linearized moment map (cf. \refeq{2spinormomentmap}) now takes the form \be \mu(\ph, \Phi_0)=\sum_{j=0}^3 \br  i e^j \ph, \Phi_0\kt e^j\otimes i\label{linearizedmoment}\ee
 
\noindent where $\{e^0=1, e^1, e^2 , e^3\}$ is an orthonormal frame of $(\Omega^0\oplus \Omega^1)(\R)$ and $i=\sqrt{-1}$ is the basis element of the Lie algebra of $U(1)$. To show Clifford multiplication respects the splitting, it suffices to show that $\ph \in S^\text{Re}$, i.e. $\mu(\ph, \Phi_0)=0$ then $\mu(e^k.\ph, \Phi_0)=0$ as well. This following from the observation that replacing $e^j$ by $e^j.e^k$ in (\refeq{linearizedmoment}) simply results in a permutation of the frame $\{e^0, e^1, e^2 ,e^3\}$. Since Clifford multiplication respects $S^\text{Re}$, it also respects the orthogonal complement $S^\text{Im}$. 

Next, we show that the connection $\nabla_{A_0}$ respects the splitting as well, i.e. $\nabla_{A_0}\ph^\text{Re}\in S^\text{Re}$ and likewise for $S^\text{Im}$. Indeed, the preceding paragraph implies that $S^\text{Re}=\{b.\Phi_0 \ | \ b \in \Omega^0(\R)\oplus \Omega^1(\R)\}$. Similarly, $S^\text{Im}=\{(ia).\Phi_0 \ | \ ia \in \Omega^0(\R)\oplus \Omega^1(i\R)\}$. Using this description, let $\ph^\text{Re}=b.\Phi_0\in \Gamma(S^\text{Re})$ be a spinor in $S^\text{Re}$. Then $$\nabla_{A_0}\ph^\text{Re}=\nabla_{A_0}(b.\Phi_0)= db.\Phi_0 + b.\nabla_{A_0}\Phi_0 \in \Gamma(S^\text{Re})$$
since $\nabla_{A_0}\Phi_0 \in S^\text{Re}$ by the Haydys Correspondence, and $db.\Phi_0=-(\star db).\Phi_0\in S^\text{Re}$ by the preceding paragraph. An identical argument applies to show that $S^\text{Im}$ is preserved as well. The decomposition \refeq{parallelsplittingCase1} therefore satisfies both hypotheses of Definition \refeq{concentratingdiracdef}. 

Writing a configuration $(\ph^\text{Re},\ph^\text{Im}, a)$ in this decomposition, the linearized Seiberg-Witten equations at $(\frac{\Phi_0}{\e}, A_0)$ take the form (\refeq{linearizationasCDO}) where
\be D=\begin{pmatrix} \slashed D_{A_0}^\text{Re} & 0 & 0  \\  0 & \slashed D_{A_0}^{\text{Im}} & 0 \\  0 &0& \bold d \end{pmatrix}\begin{pmatrix} \ph^\text{Re} \\ \ph^\text{Im} \\ a\end{pmatrix} \hspace{2cm}\mathcal A=\begin{pmatrix}0 &0 &0  \\  0&  0 & \gamma(\_){\Phi_0}\\  0 & {\mu(\_, \Phi_0)} & 0   \end{pmatrix}\begin{pmatrix} \ph^\text{Re} \\ \ph^\text{Im} \\  a\end{pmatrix}.\label{SW2spinor}\ee

\noindent{\bf Case (II):} (Flat $\text{SL}(2,\C)$ Connections on $Y^3$). In this case, the $\Z_2$-harmonic spinors are $\Z_2$-harmonic 1-forms, i.e. ones with $E=(\Omega^0\oplus \Omega^1)(\R)$. As in the previous case, the Haydys correspondence gives rise to an isomorphism $S^\text{Re}_{}\simeq E\otimes \ell_0$, which intertwines the connection induced by $A_0$ and the connection on $E\otimes \ell_0$ formed from the Levi-Civitas connection and the unique flat connection with $\Z_2$ holonomy on $\ell_0$. Again, we tacitly refer to the triple $(\mathcal Z_0, A_0, \Phi_0)$ as a $\Z_2$-harmonic 1-form. 

Let $S^\text{Re}, S^\text{Im}$ be defined by (\refeq{SReDef}). This case differs slightly from the previous one because $\gamma(-)\Phi_0: (\Omega^0\oplus \Omega^1)(\frak g_P)\to S^\text{Im}$ now has kernel. Explicitly, $\gamma(a)\Phi_0= [a \wedge \Phi_0]$ and the kernel consists of $\frak g_P$-valued forms whose $\frak g_P$-component is parallel to $\Phi_0$. Define $$\Omega^\text{Re}= \ker(\gamma(-)\Phi_0) \hspace{3cm} \Omega^\text{Im}=\ker(\gamma(-)\Phi_0)^\perp$$
\noindent then the relevant splitting is \be \frak N=S^\text{Re}\oplus \Omega^\text{Re}\hspace{3cm} \frak H=S^\text{Im} \oplus \Omega^\text{Im}. \label{SL2CSplitting}\ee

\noindent As before, we have $S^\text{Re}=\{b.\Phi_0 \ | \ b\in (\Omega^0\oplus \Omega^1)(\R)\}$, and $S^\text{Im}=\{\gamma(a)\Phi_0 \ | \ a \in \Omega^\text{Im} \}$.  The proof that $\nabla_{A_0}$ respects $S^\text{Re}$ follows by the same argument as in Case (i), and differentiating the orthogonality relation $0=\br\ph^\text{Re}, \ph^\text{Im} \kt$ shows $\nabla_{A_0}$ preserves $S^\text{Im}$ as well. In this case, since $S=(\Omega^0\oplus \Omega^1)(\frak g_P)$,  and $\mu(-,\Phi_0)=[-\wedge \Phi_0]=\gamma(-)\Phi_0$ one secretly has $\Omega^\text{Re}=S^\text{Re}$ and likewise for the imaginary components; thus the splitting of the form components, being the same splitting, is also parallel. 

By the description of $S^\text{Im}$ following (\refeq{SL2CSplitting}), the image of $\gamma(-)\Phi_0$ lies in $S^\text{Im}$. Similarly, $\mu(-,\Phi_0)$ has image only in $\Omega^\text{Im}$ since $\br  a^\text{Re}, \mu(\ph,\Phi_0)\kt=\br \gamma(a^\text{Re})\Phi_0, \ph\kt=0$ by definition of $\Omega^\text{Re}$. Writing a configuration $(\ph^\text{Re}, a^\text{Re},\ph^\text{Im}, a^\text{Im})$ in the decomposition  (\refeq{SL2CSplitting}), the linearized Seiberg-Witten equations at $(\frac{\Phi_0}{\e}, A_0)$ take the form (\refeq{linearizationasCDO}) where
\be D=\begin{pmatrix} \slashed D_{A_0}^\text{Re} & 0 &0 & 0  \\  0& \bold d_{A_0}^\text{Re} & 0 & 0 \\  0 & 0 &\slashed D_{A_0}^\text{Im}&0\\ 0 & 0 &0 & \bold d_{A_0}^\text{Im}\end{pmatrix} \begin{pmatrix} \ph^\text{Re} \\ a^\text{Re} \\ \ph^\text{Im} \\  a^\text{Im}\end{pmatrix}\hspace{2cm}\mathcal A=\begin{pmatrix}0 & 0& 0 &0  \\  0 &0 &0 &0 \\ 0 & 0 &0 & \gamma(\_){\Phi_0} \\  0 &0 &{\mu(\_, \Phi_0)} & 0   \end{pmatrix}\begin{pmatrix} \ph^\text{Re} \\ a^\text{Re} \\ \ph^\text{Im} \\  a^\text{Im}\end{pmatrix}\label{FlatSL2C}\ee

\noindent{\bf Case (III):} (Two spinor Seiberg-Witten on $X^4$). The four-dimensional case for the two-spinor Seiberg-Witten equations is virtually identical to the three-dimensional case: let $(S^+)^\text{Re}= \ker(\mu(-,\Phi_0))$ as in (\refeq{SReDef}). Set 

\be \mathfrak N^+= (S^+)^\text{Re} \hspace{3cm}  \frak H^+:=(S^+)^\text{Im} \oplus  \Omega^1(i\R).\label{CaseIIISplitting1}\ee

\noindent Then, let $(S^-)^\text{Re}= \{ \alpha. \ph^\text{Re} \ | \ \alpha \in \Omega^1(\R) \ , \ \ph \in (S^+)^\text{Re}\}$ and likewise for $(S^-)^\text{Im}$, and define 

\be \mathfrak N^-= (S^-)^\text{Re} \hspace{3cm}  \frak H^-:=(S^+)^\text{Im} \oplus  (\Omega^0 \oplus \Omega^2_+)(i\R).\label{CaseIIISplitting2}\ee

\noindent The proof that the fixed degeneracy hypothesis is satisfied carries over {\it mutatis mutandis} from the 3-dimensional case, and the expressions (\refeq{SW2spinor}) are identical regarded as matrices for the splittings (\refeq{CaseIIISplitting1}) and (\refeq{CaseIIISplitting2}).

  \medskip 
  
\noindent{\bf Case (IV):} (Complex $ASD$ Equations on $X^4$). This case is the four-dimensional version of Case (II) in the same way that  Case (IV) is the four-dimensional version of Case (I).

\end{proof}

\begin{rem}
To elaborate on Remark \ref{rem1.6}, the requirement that the splitting be parallel is the  barrier to extending Theorem \ref{mainb} to other generalized Seiberg-Witten equations (e.g. the Seiberg-Witten equations with $r>2$ spinors). In such cases, $\mu^{-1}(0)$ is a not simply the quotient of a vector space by a finite group, and the proof that the splitting is parallel breaks down.  (In such cases, and the limiting Fueter equation has no interpretation as a linear Dirac equation). Analytically, the failure of the splitting to be parallel causes the assertion (\refeq{parallelD}), as it would in general involve terms including the covariant derivative of $q_0$ (cf \cite{DWDeformations} Section 6.2).  
\end{rem}

The next lemma verifies the first two hypotheses of Corollary \ref{cornonlinear} apply in the case of generalized Seiberg-Witten equations. The third hypothesis, i.e. the estimates (\refeq{quadraticestimates}), are the subject of the upcoming Lemmas \ref{nonlinear3d} and \ref{nonlinear4d} in the next section.

\begin{lm}\label{formofQ} 
The generalized Seiberg-Witten equations in Cases (I) --  (VI) can be written in the form \be (D+ \tfrac{1}{\e}\mathcal A)\frak q + Q(\frak q)=f\label{nonlinearDirac} \ee 

\noindent where  $\frak q=(\ph,a)$ and $f$ and $Q$ satisfy the following:  
\begin{itemize}
\item[(i)] $f\in \Gamma(\frak N)$
\item[(ii)] $Q(\frak q)=Q_1(\frak q)\pi_\frak H(\frak q)$ for a linear operator $Q_1$ 
\item[(iii)] $Q_1^\star \sigma_D(\xi)=\sigma_D(\xi)^\star Q_1 $ for all $\xi \in T^\star Y$ (resp. $T^\star X$).  
\end{itemize}
\end{lm}

\begin{proof}
The (un-renormalized) $\Z_2$-harmonic spinor satisfies the equations $$\slashed D_{A_0}\left(\tfrac{\Phi_0}{\e}\right)=0 \hspace{2cm} \tfrac{1}{2}\tfrac{\mu(\Phi_0,\Phi_0)}{\e^2}=0$$
on $Y-\mathcal Z_0$ (resp. $F_{A_0}^+ , \slashed D_{A_0}^+$ on $X-\mathcal Z_0$). In particular, $(\tfrac{\Phi_0}{\e}, A_0)$ fails to solve the Seiberg-Witten equations on $Y-\mathcal Z_0$ (resp. $X-\mathcal Z_0$) by $\star F_{A_0}$ (resp. $F_{A_0}^+$). 
 Thus $(\tfrac{\Phi_0}{\e}, A_0)+(\ph,a)$ satisfies the Seiberg-Witten equations if and only if $(\ph,a)$ satisfies the deformation equation (cf. (\refeq{deformationeq}) \be \Big(\mathcal L_{\left(\tfrac{\Phi_0}{\e}, A_0\right)}+ Q\Big) (\ph,a)= -E_0. \label{nonlineardeformation}\ee
\noindent where $Q(\ph,a)=(\gamma(a)\ph \ , \  \mu(\ph,\ph))$, and $E_0=-\star F_{A_0}$ (resp. $-F_{A_0}^+$). The Haydys Correspondence implies that $F_{A_0}\in \Omega^\text{Re}$  (see Appendix C of \cite{DWAssociatives}
), hence $E_0\in \Gamma(\mathfrak N)$ and statement (i) is satisfied. The proof now proceeds in each of the cases individually. 

\medskip

\noindent{\bf Case (I):} (Two Spinor  Seiberg-Witten on $Y^3$). As in  Case (I) in the proof of Lemma \ref{fixeddegeneracyL}, Clifford multiplication by $\R$-valued forms preserves the splitting $S=S^\text{Re}\oplus S^\text{Im}$ while Clifford multiplication by $\frak g_P=i\R$-valued forms reverses it. Moreover, in this case, the Haydys Correspondence implies that $F_{A_0}=0$. Thus in the splitting (\refeq{parallelsplittingCase1}) of Lemma \ref{fixeddegeneracyL} the non-linear deformation equation (\refeq{nonlineardeformation}) on triples $(\ph^\text{Re},\ph^\text{Im},a)$ takes the form 

\begin{eqnarray}
\underbrace{\begin{pmatrix}
\slashed D^\text{Re}_{A_0}\ph^{\text{Re}}  \\ 
\slashed D^\text{Im}_{A_0}\ph^{\text{Im}}  +\gamma(a)\frac{\Phi_0}{\e}\\
\bold d a   \ + \  \tfrac{\mu(\ph^\text{Im}, \Phi_0)}{\e}
\end{pmatrix}}_{\mathcal L_{(\Phi_0, A_0})}
 \ \ +  \ \ 
 \underbrace{\begin{pmatrix}
 \gamma(a)\ph^{\text{Im}} \\
 \gamma(a)\ph^{\text{Re}} \\
 2\mu(\ph^{\text{Im}},\ph^{\text{Re}}
 \end{pmatrix}}_{Q(\ph,a)} \ \ = \ \ \begin{pmatrix} 0 \\ 0 \\ 0 \end{pmatrix}.
 \label{3equationversion}
 \end{eqnarray}

\noindent Note that each term of $Q(\ph,a)$  contains at least a linear factor in $(\ph^\text{Im},a)$ which is the assertion of statement (ii) of the lemma. The equations can therefore be written as: 

\begin{eqnarray}
\left( \begin{pmatrix} \slashed D_{A_0}^\text{Re}   & 0 &0  \\  0 & \slashed D_{A_0}^\text{Im} & 0  \\  0 & 0 & \bold d
 \end{pmatrix} +  \frac{1}{\e}\begin{pmatrix}  0  &  0 & \gamma(\_)\e\ph^\text{Im} \\  0 &  0 & \gamma(\_)(\Phi_0 + \e \ph^{\text{Re}} )\\  0& \mu(\_ , \Phi_0 + \e\ph^\text{Re})&  0
  \end{pmatrix} \right)\begin{pmatrix} \ph^\text{Re} \\ \ph^{\text{Im}} \\ a\end{pmatrix} =0.\label{nonlinearconcentrating}
\end{eqnarray} 

\noindent which has the form  (\refeq{nonlinearDirac}). Equivalently, the equation has been recast as a concentrating Dirac operator with $\mathcal A$ in the form (\refeq{fixeddegeneracy2}) with 
\be \begin{pmatrix} 0 &0  \\ 0 & A_{\mathfrak H}\end{pmatrix}=\begin{pmatrix}  0 & 0& 0 \\ 0 & 0& \gamma(\_)\Phi_0 \\ 0 & \mu(\_,\Phi_0) &0  \end{pmatrix}  \hspace{2cm} A_\e=\begin{pmatrix}  0 & 0& \gamma(\_)\e \ph^\text{Im} \\ 0 & 0& \gamma(\_) \e \ph^\text{Re}\\ \mu(\_, \e \ph^\text{Im}) & \mu(\_,\e \ph^\text{Im}) &0  \end{pmatrix}.\label{Aepsilon1}\ee

\noindent where $A_\frak H$ is the lower $2\times 2$ block. In this form, it is now obvious that item (iii) holds by Lemma \ref{concentrationL}, (as the proof of the commutation relation applied for any spinor).  
\medskip 

\noindent{\bf Case (II):} (Flat $\text{SL}(2,\C)$ Connections on $Y^3$) The only salient difference between this case and the previous one is that there is an additional non-linear term $a\wedge a$ arising from the non-abelianness of the connection. In this case, item (ii) of the lemma follows from the following observation: since $a \in \Omega^\text{Re}, \ph \in S^\text{Re}$ means that that $a,\ph$ are $\frak g_P$-valued 1-forms whose Lie algebra component is parallel to $\Phi_0$, the commutator $\gamma(a^\text{Re})\ph^\text{Re}=[a^\text{Re}\wedge \ph^\text{Re}]=0$. Likewise, $\mu(\ph^\text{Re}, \ph^\text{Re})=[\ph^\text{Re}\wedge \ph^\text{Re}]=0$ and $a^\text{Re} \wedge a^\text{Re}=0$. Consequently,   (\refeq{nonlineardeformation}) in this case takes the form

\begin{eqnarray}
\slashed D_{A_0}^\text{Re}\ph^{\text{Re}}   \hspace{1.0cm} \ + \ \hspace{1.75cm} \    \Pi^\text{Re} (\gamma(a^\text{Re})\ph^\text{Im}  \ + \ \gamma(a^\text{Im})\ph^\text{Re} \ + \ \gamma(a^\text{Im})\ph^\text{Im}  )&=& 0\nonumber \\ 
\bold d_{A_0}^\text{Re}a^{\text{Re}}   \hspace{1.0cm} \ + \  \Pi^\text{Re} (2\mu(\ph^\text{Re},\ph^\text{Im}) \ + \ \mu(\ph^\text{Im}, \ph^\text{Im}) \  +  \   (a^\text{Re}+ a^\text{Im})\wedge (a^\text{Im} )  & = & -\star F_{A_0} \nonumber \\ 
\slashed D_{A_0}^\text{Im}\ph^{\text{Im}} +  \gamma(a)\tfrac{\Phi_0}{\e}  \ \ \ + \ \hspace{1.75cm}   \   \Pi^\text{Im} (\gamma(a^\text{Re})\ph^\text{Im}  \ + \ \gamma(a^\text{Im})\ph^\text{Re} \ + \ \gamma(a^\text{Im})\ph^\text{Im}  )&=& 0 \nonumber \\ 
\underbrace{\bold d_{A_0}^\text{Im}a^{\text{Im}}  +  \tfrac{\mu(\ph^\text{Im},\Phi_0)}{\e}}_{\mathcal L_{(\Phi_0, A_0)}} \ + \  \underbrace{\Pi^\text{Im} (2\mu(\ph^\text{Re},\ph^\text{Im}) \ + \ \mu(\ph^\text{Im}, \ph^\text{Im}) \  +  \   (a^\text{Re}+ a^\text{Im})\wedge (a^\text{Im} ) }_{Q(\ph,a)} & = & 0.
\label{3equationversion}
\end{eqnarray}

\noindent Thus, with $ A_\frak H$ being the lower block of (\refeq{FlatSL2C}) and 

\be
A_\e =\begin{pmatrix} 
0 & \Pi^\text{Re}(\gamma(\_) \e\ph^\text{Im})& 0& \Pi^\text{Re}(\gamma(\_) \e\ph) \\ 
 \Pi^\text{Re}(\mu(\_, \e\ph^\text{Im})) & \Pi^\text{Re}(\_ \wedge \e a )&  \Pi^\text{Re}(\mu(\_, \e\ph)) & 0  \\ 
0 & \Pi^\text{Im}(\gamma(\_) \e\ph^\text{Im})& 0&  \Pi^\text{Im}(\gamma(\_) \e\ph) \\ 
 \Pi^\text{Im}(\mu(\_, \e\ph^\text{Im}))  & 0&  \Pi^\text{Im}(\mu(\_, \e\ph)) & \Pi^\text{Im}(\_  \wedge \e a ) 
\end{pmatrix} \begin{pmatrix} \ph^\text{Re} \\ a^\text{Re} \\ \ph^\text{Im} \\ a^\text{Im} \end{pmatrix}
\label{Aepsilon2}
\ee
the operator has the desired form. Item (iii) of the lemma follows identically to the previous case. 
\medskip 

\noindent{\bf Cases (III) -- (IV):} These cases are analogous to Cases (I) and (II) in the same way as Lemma \ref{fixeddegeneracyL}. 
\medskip

\end{proof}
 
\section{Bootstrapping}
\label{section5}
The convergence $\Phi_i\overset{L^{2,2}} \rightharpoonup \Phi_0$ and $A_i\overset{L^{1,2}} \rightharpoonup A_0$ obtained in Theorem \ref{compactness} cannot be naively bootstrapped to higher Sobolev spaces using elliptic estimates in the standard way. Indeed, if  it is known that  $A_i \to A_0, \Phi_i \to \Phi_0$ in $L^{k,2}$ on a compact subset $K'\Subset Y-\mathcal Z_0$, or equivalently using the notation of (\refeq{deformationeq}), $a_i \ , \ \e_i\ph_i\to 0$ in $L^{k,2}$. Recalling that Theorem \ref{compactness} asserts convergence for the renormalized spinor $\Phi_i=\Phi_0+\e_i\ph_i$, then the elliptic estimate on a compact subsets $K\Subset K'\Subset Y-\mathcal Z_0$ reads: 

$$\|a_i\|_{L^{k+1,2}(K)} \leq C_k \left(\frac{1}{\e_i^2} \|\mu(\Phi_0, \e_i\ph_i) + \mu(\e_i\ph_i, \e_i \ph_i)\|_{L^{k,2}(K')} + \|a\|_{L^{k,2}(K')}\right). $$

\noindent Because of the factor of $\e_i^{-2}$ on the right hand side, convergence $a_i\to 0$ in  $L^{k+1,2}$ does not follow. Concluding convergence in this way requires knowing that $\ph_i\to 0$ in $L^{k,2}$ at least as fast as $\e_i^2$. The exponential convergence furnished by Corollary \ref{cornonlinear} overcomes this issue and allows the bootstrapping to proceed.

Lemmas \ref{concentrationL} and \ref{fixeddegeneracyL} show that the assumptions for Theorem \ref{maina} are satisfied for the linearized Seiberg-Witten equations in all case (I)--(IV). The next two lemmas show that the non-linear terms satisfy the final assumption  \refeq{quadraticestimates} the necessary to apply Corollary \ref{cornonlinear} in the 3-dimensional and 4-dimensional cases respectively. First, we note the following fact:

\begin{lm}\label{FA0bound} Suppose $(\Phi_i, A_i)\to (\mathcal Z_0, \Phi_0, A_0)$ is a sequence of generalized Seiberg-Witten solutions converging to a $\Z_2$-harmonic spinor in the sense of Theorem \ref{compactness}. The limiting connection $A_0$ in dimensions $n=3$ and $n=4$ satisfies $$ \|F_{A_0}\|_{C^2(K')}<\infty. $$ 
\end{lm}

\begin{proof} In Cases (I) and (III), the Haydys Correspondence implies that $F_{A_0}=0$, and the result is immediate. 

In Cases (II) and (IV), the limiting curvature $F_{A_0}$ in Theorem \ref{compactness} is harmonic, i.e. \be (d_{A_0}+d_{A_0}^\star)F_{A_0}=0 \label{harmoniccurvature}\ee (in fact, Remark 1.37 of \cite{WalpuskiZhangCompactness} shows $F_{A_0}$ vanishes in Case (II)). Because $A_0$ is only known to be $L^{1,2}$, the coefficients of (\refeq{harmoniccurvature}) are too rough to initiate a standard bootstrapping argument. However, there is more information in this case: the Haydys Correspondence with stabilizers (see Appendix C of \cite{DWAssociatives}) shows that $F_{A_0}\in \Omega^\text{Re}$. Since $A_0$ respects the splitting $\Omega^\text{Re}\oplus \Omega^\text{Im}$ by Proposition \ref{fixeddegeneracyL}, (\refeq{harmoniccurvature}) restricts to an equation on the smooth bundle $\Omega^\text{Re}\to Y-\mathcal Z_0$ (resp. $X-\mathcal Z_0$). 

As in the proof of Proposition \ref{fixeddegeneracyL}, one has $\Omega^\text{Re} \simeq S^\text{Re}\simeq E \otimes_\R \ell_0$, where $E=\Lambda^1(\R)$ is a rank 4 real Clifford module on $Y$ (resp $X$), and $\ell_0 \to Y-\mathcal Z_0$ (resp. $X-\mathcal Z_0$) is real line bundle. Moreover, (cf. the discussion after Definition \ref{Z2Harmonic}), the restriction of $A_0$ viewed as a connection on $E\otimes_\R \ell_0$ via this isomorphism is the connection arising from a smooth spin connection on $E$ and the unique flat connection with holonomy in $\Z_2$ on $\ell_0$. Thus (\refeq{harmoniccurvature}) becomes $$(d_\Gamma + d_\Gamma^\star)F_{A_0}^\text{Re}=0$$
\noindent where $d_\Gamma$ is (up to isomorphism), the exterior covariant derivative on $\Lambda^\bullet (\R)\otimes \ell_0$. By elliptic regularity, and since $F_{A_0}=F_{A_0}^\text{Re}$, one has $\|F_{A_0}\|_{L^{k,2}}\leq C_k \|F_{A_0}\|_{L^2}<\infty$ for every $k$, and the conclusion follows from the Sobolev embedding theorem.

\end{proof}

\bigskip 

\begin{lm} \label{nonlinear3d}Let $Y$ be a closed, oriented 3-manifold and suppose that $(\Phi_i, A_i, \e_i)\to (\mathcal Z_0, A_0,\Phi_0)$ is a sequence of solution to (\refeq{bSW1}--\refeq{bSW3}) converging to a $\Z_2$-harmonic spinor in the sense of Theorem \ref{compactness} in Case (I) or Case (II). Let $\Phi_i=\Phi_0 + \e_i \ph_i$ and $A_i = A_0 + a_i$. Then \be \|A_\e\|_{L^{1,3}(K)} \to 0  \hspace{1cm}\text{and }\hspace{1cm} \|A_\e\|_{C^{0}(K)} \to 0\label{EqLem6.2}\ee
on compact subsets $K\Subset Y-\mathcal Z_0$, where $A_\e$ is the matrix (\refeq{Aepsilon1}, \refeq{Aepsilon2}) from the proof of Lemma \ref{formofQ}.  
\end{lm}

\begin{proof}  In Case (I), the matrix $A_\e$ of (\refeq{Aepsilon1}) contains only terms involving $\e_i\ph_i$. In this case, the conclusion is immediate from Theorem \ref{compactness} and the Sobolev embedding; indeed, the conclusion of Theorem \ref{compactness} shows that $\e_i\ph_i= \Phi_i - \Phi_0 \to 0$ in $L^{2,2}$. In particular, $\e_i \ph_i\to 0$ in $L^{1,6}$, thus {\it a fortiori} in $L^{1,3}$ and in $C^0$ by the embedding $C^{0,\alpha}\hookrightarrow L^{1,6}$.  In Case (II), the matrix $A_\e$ from (\refeq{Aepsilon2})  includes entries involving $\e_i \ph_i$, which converge as above, and also entries of the form $\e_i a_i$. To prove the lemma, it therefore suffices to show that $\|\e_i a_i\|_{L^{1,p}}\to 0$ for any $p>3$.

For the remainder of the proof, the subscript $i$ is kept implicit in the notation. Letting $\psi=\e \ph$ be the re-normalized deformation of the spinor, the curvature equation for the deformation $a$ reads:  

\be \star F_{A_0}+\bold d_{A_0} a =\frac{1}{\e^2}\mu(\Phi_0+\psi, \Phi_0 + \psi)- a\wedge a .\label{equationforelliptic}\ee

\noindent where $\bold d_{A_0}$ is the Dirac operator from Example \ref{case2eg}. The elliptic estimate for $\bold d_{A_0}$ applied to $a$ yields 
\begin{eqnarray}
\|a\|_{L^{1,p}}&\leq& C_{k,p} \left(\frac{1}{\e^2}\|\mu(\Phi_0, \psi)\|_{L^{p}}+\frac{1}{\e^2}\|\mu(\psi, \psi)\|_{L^{p}} + \|a\wedge a\|_{L^{p}} + \|F_{A_0}\|_{L^p}+ \|a\|_{L^2} \right)\label{elliptic1}
\end{eqnarray}

\noindent holds for $p\geq 2$. Differentiating (\refeq{equationforelliptic}) and commuting covariant derivatives yields the estimate \begin{eqnarray}
\|\nabla_{A_0}a\|_{L^{1,p}}&\leq& C_{k,p} \Big(\frac{1}{\e^2}\|\nabla_{A_0}\mu(\Phi_0, \psi)\|_{L^{p}}+\frac{1}{\e^2}\|\nabla_{A_0}\mu(\psi, \psi)\|_{L^{p}} + \|\nabla_{A_0}a\wedge a\|_{L^{p}} \nonumber \\  & & \ \  \  \ + \  \|\nabla_{A_0}F_{A_0}\|_{L^p}  \ + \  \| \ [F_{A_0}\wedge a] \ \|_{L^p} + \|a\|_{L^p} \Big)\label{elliptic3}.
\end{eqnarray}

\noindent for the covariant derivative (commuting covariant derivatives gives rise to the curvature term $[F_{A_0}\wedge a]$ along with a bounded Riemannian curvature term which has been absorbed into the final term of (\refeq{elliptic3})).

 Now we bootstrap. To begin, we know that $\Phi_0 \in L^{2,2}$, $\psi=\e \ph\to 0$  in  $L^{2,2}$, and $a\to 0$ in $L^{1,2}$.  

\noindent {\it Step 0:} By the Sobolev embeddings $L^{1,2}\hookrightarrow L^{6}$, and $L^{2,2}\hookrightarrow L^{2,2}$, the following quantities are bounded 

\indent \indent \  uniformly in $\e$:  $\|a\|_{L^6}$ \ , \ $\|\mu(\Phi_0, \psi)\|_{L^3}$ \ , \ $\|\mu(\psi,\psi)\|_{L^3}$. Additionally, $\|F_{A_0}\|_{L^3}$ is bounded by 

\indent \indent \ Lemma \refeq{FA0bound}. 

\noindent {\it Step 1:} Apply the elliptic estimate (\refeq{elliptic1}) with $p=3$ to conclude that $\|\e^2a\|_{L^{1,3}}$ is bounded in $L^{1,3}$. 

\noindent {\it Step 2:} By H\"older's inequality with $p=3/2$ and $q=3$,  $$\|\nabla_{A_0}a\wedge a\|_{L^2} \leq \|a\|_{L^{1,3}} \|a\|_{L^6} $$

\indent  \indent  \ \ is bounded. Likewise for $\|\mu(\nabla \Phi_0, \psi) \|_{L^2},\|\mu( \Phi_0, \nabla \psi) \|_{L^2}, \|\mu(\nabla \psi, \psi) \|_{L^2}$ using the bounds on

\indent \indent \ \  $\Phi_0, \psi$ from Theorem \ref{compactness}. 

\noindent {\it Step 3:} Using Lemma \ref{FA0bound} to bound the terms involving $F_{A_0}$, apply the elliptic estimate (\refeq{elliptic3}) with  \indent \indent \ \  $p=2$ to conclude that $\|\e^2a\|_{L^{2,2}}$ is bounded. 

\noindent {\it Step 4:}  By Cauchy-Schwartz,  $$\|\nabla_{A_0}a\wedge a\|_{L^{5/2}} \leq \|a\|_{L^{1,5}} \|a\|_{L^5} $$

\indent  \indent  \ \ and likewise for $\|\mu(\nabla \Phi_0, \psi) \|_{L^{5/2}},\|\mu( \Phi_0, \nabla \psi) \|_{L^{5/2}}, \|\mu(\nabla \psi, \psi) \|_{L^{5/2}}$ using the bounds on

\indent \indent \ \  $\Phi_0, \psi$ from Theorem \ref{compactness}.

\noindent {\it Step 5:} Using Lemma \ref{FA0bound} to bound the terms involving $F_{A_0}$, apply the elliptic estimate (\refeq{elliptic3}) with \indent \indent \ \  \indent \indent \ \ \  $p=5/2 $ to conclude that $\|\e^2a\|_{L^{2,5/2}}$ is bounded.

\noindent {\it Step 6:} Applying the interpolation inequality with $\delta<<1$, 

$$\|a\|_{L^{1,p}}\leq C_K \left( \|a\|_{L^{2,5/2}}^{1/2} \|a\|^{1/2}_{L^{6-\delta}} + \|a\|_{L^2}\right)$$

\indent \indent \ \ \ where $\frac{1}{p}=\frac{1}{3}+\frac{1}{2}\left(\frac{1}{5/2}- \frac{2}{3}\right) + \tfrac{1-1/2}{6-\delta}$ to conclude $\|\e a\|_{L^{1,p}}\to 0$ for $p=\tfrac{60}{17}-\delta'>3$. 

\medskip
 
\noindent It follows that $\e_i a_i\to 0$ in $L^{1,3}$ and $C^0$ by the Sobolev embedding since $K$ is compact.
\end{proof}

\bigskip

\begin{lm} \label{nonlinear4d}Let $X$ be a closed, oriented 4-manifold and suppose that $(\Phi_i, A_i, \e_i)\to (\mathcal Z_0, A_0,\Phi_0)$ is a sequence of solution to (\refeq{bSW1}--\refeq{bSW3}) converging to a $\Z_2$-harmonic spinor in the sense of Theorem \ref{compactness} 
in Case (III) or Case (IV). Let $(\ph_i, a_i)$ denote the difference from $(\Phi_0, A_0)$ as in (\refeq{deformationeq}). Then \be \|A_\e\|_{L^{1,4}(K)} \to 0   \hspace{1cm}\text{and }\hspace{1cm} \|A_\e\|_{C^{0}(K)} \to 0 \label{nonlinear4dbounds}\ee
on compact subsets $K\Subset Y-\mathcal Z_0$, where $A_\e$ is the four-dimensional analogue of the matrices (\refeq{Aepsilon1}, \refeq{Aepsilon2}) from the proof of Lemma \ref{formofQ} in Cases (III) and (IV) respectively.  . 
\end{lm}

\begin{proof} In Case (III) the conclusion again follows directly from Theorem \ref{compactness} and the Sobolev embedding. Case (IV) follows from a similar bootstrapping procedure as in the previous lemma, but repeatedly applying Steps 1--5 until convergence in $L^{2,4+\gamma}$ of $\e^2 a_i$ is obtained for $0<\gamma<<1$. Then, interpolating as in Step 6 of Lemma \ref{nonlinear3d} establishes that $\e a_i\to 0$ in $L^{1,4+\gamma}$. The second bound of (\refeq{nonlinear4dbounds}) follows from the Sobolev embedding $L^{1,4+\gamma}\hookrightarrow  C^{0,\alpha}$. 
\end{proof}

The convergence to a $\Z_2$-harmonic spinor now, by design, fits into the abstract framework of Sections \ref{section2}--\ref{section3}. The following proposition retains the setting of the previous lemmas, and additionally uses the notation of $S^\text{Im}, \Omega^\text{Im}$ from Section \ref{section4.4}. 

\begin{prop}\label{exponentialdecay}
Suppose that $(\Phi_i, A_i,\e_i)\to (\mathcal Z_0, A_0, \Phi_0)$ is a sequence of solutions to (\refeq{bSW1}--\refeq{bSW3}) converging in the sense of Theorem \ref{compactness} in Cases (I)--(III);  in Case (IV) further assume that $A_i\to A_0$ in $L^{1,p}_{loc}$ and $\Phi_i\to \Phi_0$ in $L^{2,p}_{loc}$ for $p>2$.  Let $(\ph_i, a_i)=(\ph^\text{Re}, \ph^\text{Im}, a^\text{Re}, a^\text{Im})$  \footnote{In Cases (I) and (III), $\Omega^{\text{Re}}$ is empty hence $a=a^\text{Im}$.} denote the perturbations from the limiting data which satisfy (\refeq{deformationeq}). Then there exist constants $C,c$ depending only on $\Phi_0$ and background data such that 

\be \|(\ph_i^\text{Im}, a_i^\text{Im})\|_{C^0(K)} \leq \frac{C}{R_K^{n/2}}  \text{Exp}\left( - \frac{c R_K}{\e_i}\right) \|(\ph_i,a_i)\|_{L^{1,2}(K')}\ee

\noindent on compact subsets $K\Subset K'\Subset Y-\mathcal Z_0$ (resp. $X-\mathcal Z_0$), where $n=3,4$ is the dimension in the respective cases, and $R_K=\text{dist}(K,\mathcal Z_0)$.  

\end{prop}

\begin{proof}

The regularity on $A_\e$ established in Lemmas \ref{nonlinear3d}--\ref{nonlinear4d} is sufficient to apply Proposition 6.1 of \cite{DWDeformations} to show that there is a gauge transformation putting the deformation $(\ph_i,a_i)$ in the gauge  (\refeq{gaugefixing}) (note that \cite{DWDeformations}, Proposition 6.1) applies identically in 4-dimensions for the equivalent Sobolev range). As in Claim \ref{FA0bound}, the limiting connection $F_{A_0}\in \Omega^\text{Re}$, so the limiting configuration solves the Seiberg-Witten equations in the $S^\text{Im},\Omega^{Im}$ components (i.e. $SW(\Phi_0, A_0)=-F_{A_0} \in \Omega^\text{Re}$ as in \refeq{deformationeq}). Thus In this gauge, Proposition \ref{concentrationL},  Lemma \ref{fixeddegeneracyL}, and Lemmas \ref{nonlinear3d}--\ref{nonlinear4d} show that the assumption of Corollary \ref{cornonlinear} hold in the respective cases.  The conclusion then follows directly. 
\end{proof}

Using the exponential convergence of Proposition \ref{exponentialdecay}, we may now conclude the proof of Theorem \ref{mainb}: 

\begin{proof}[Proof of Theorem \ref{mainb}] Let $(\Phi_0+\psi_i, a_i)=\left(\e_i \left(\frac{\Phi_0}{\e_i}+\ph_i\right), a_i\right)$ denote the re-normalized sequence. For the remainder of the proof, the subscript is kept implicit. 
Let $K$ denote a compact subset, and let $K_0 \Supset K_{1}\Supset K_2\ldots$ be a nested sequence of compact subsets of $Y-\mathcal Z_0$ so that $K\subseteq {\bigcap K_j}$. By the assumption that the sequence converges in the sense of Theorem \ref{compactness}, Proposition \refeq{exponentialdecay} applies to show

$$\|(\ph^{\text{Im}},a)\|_{C^0(K_0)}\leq C_K\text{Exp}\left(-\frac{c}{\e}\right)\|(\ph^{\text{Re}}, \ph^{\text{Im}}, a)\|_{L^{1,2}(Y-\mathcal Z_0)}\leq \frac{C}{\e}\text{Exp}\left(-\frac{c}{\e}\right).$$

\noindent It now follows from the curvature equation $$F_{A_i} + \frac{1}{2\e} \mu\left(\ph_i^\text{Im}, \Phi_0\right) + \frac{1}{2}\mu(\ph_i^\text{Re},\ph_i^\text{Im})=0$$

\noindent that $F_{A_0}=0$ (resp. $F_{A_0}^+=0$) for the limiting connection. 

Bootstrapping now proceeds in the standard way using the Seiberg-Witten equations (\refeq{SW1}--\refeq{SW2}) to show 
\bea
\|(\psi_i^{\text{Im}},a_i^\text{Im})\|_{L^{k,2}(K_k)}&\leq& \frac{C_k}{\e^{2k+1}}\text{Exp}\left(-\frac{c}{\e}\right)\to 0\\
\|(\psi_i^{\text{Re}}, a_i^\text{Re})\|_{L^{k,2}(K_k)}&\to &  0
\eea
for every $k\geq 3$ (thus in particular on $K\subseteq K_k$), and the conclusion follows.

 To spell out the first few steps for $n=3$ in Case (I), first apply the elliptic estimate for the operator $\bold d$ of (\refeq{4.7}), which in this case is independent of $A$ with $(k,p)=(1,6)$. The second equation (\refeq{SW2}) shows

\bea
\|a_i\|_{L^{1,6}(K_1)}&\leq& C_{1,6}\left(\frac{1}{\e^2}\|\mu(\Phi_0, \psi_i^\text{Im}) + \mu(\psi_i^\text{Re}, \psi_i^{\text{Im}})\|_{L^6(K_0)} + \|a_i\|_{L^6(K_0)}\right) \\
&\leq & \frac{C}{\e^{3}} \ \text{Exp}\left(-\frac{c}{\e}\right)\to 0.
\eea

\noindent \noindent This in turn implies $A_0 \in L^{1,6}$. Likewise, for the spinor, applying the elliptic estimate for $\slashed D_{\widetilde A}$ (for a fixed smooth background connection $\widetilde A$) to the first equation (\refeq{SW1}) yields
\bea
\|\psi^\text{Im}_i\|_{L^{1,6}(K_1)}&\leq& C_{1,6}\left(\| (\widetilde A-A_0)\psi_i^\text{Im}\|_{L^6(K_0)} + \| \gamma(a_i)\psi_i^\text{Re}\|_{L^6(K_0)}  + \|\psi_i^{\text{Im}}\|_{L^6(K_0)}\right) \\
&\leq & \frac{C}{\e^{3}}\text{Exp}(-\frac{c}{\e})\to 0\\
\|\psi^\text{Re}_i\|_{L^{1,6}(K_1)}&\leq& C_{1,6}\left(\| (\widetilde A-A_0)\psi_i^\text{Re}\|_{L^6(K_0)} + \| \gamma(a_i)\psi_i^\text{Im}\|_{L^6(K_0)}  + \|\psi_i^{\text{Re}}\|_{L^2(K_0)}\right)\to 0
\eea

Repeating the elliptic estimate for the connection now with $(k,p)=(2,2)$ and using the fact that multiplication induces a bounded map $L^{1,6}\times L^{1,6}\to L^{1,2}$ one has 
\bea
\|a_i\|_{L^{2,2}(K_2)}&\leq& C_{1,6}\left(\frac{1}{\e^2}\|\mu(\Phi_0, \psi_i^\text{Im}) + \mu(\psi_i^\text{Re}, \psi_i^{\text{Im}})\|_{L^{1,2}(K_1)} + \|a_i\|_{L^6(K_1)}\right) \\
&\leq & \frac{C_{2,2}}{\e^2}\left( \|\Phi_0\|_{L^{1,6}(K_1)} \|\psi_i^\text{Im}\|_{L^{1,6}(K_1)} +\|\psi_i^\text{Re}\|_{L^{1,6}(K_1)} \| \psi_i^{\text{Im}}\|_{L^{1,6}(K_1)} + \|a_i\|_{L^{1,2}(K_1)}\right) \\
&\leq & \frac{C_2}{\e^{5}}\text{Exp}\left(-\frac{c}{\e}\right)\to 0.
\eea
  
  \noindent The bootstrapping continues in this fashion to obtain the desired bounds in $L^{k,2}$ for every $k>2$. 
\end{proof}

\begin{rem} Theorem \ref{mainb} does not applying in the case of the $\text{ADHM}_{1,2}$ Seiberg-Witten equations described in Example \ref{ADHMeg}. 
As discussed in remark \ref{rem1.6} the form of the non-linear terms in this case does not satisfy the hypotheses of Corollary \ref{cornonlinear}. Specifically, in this case one need not have that $\mu(\Psi^\text{Re}, \Psi^\text{Re})=0$, thus the differential inequality \refeq{improveddifferential} contains an additional term of the form $\br q_1, f(q_0, q_0)\kt$. It seems likely to the author that in this situation, the techniques of \cite{DWExistence} (Section 5) could be extended to bootstrap convergence to $C^\infty_{loc}$ (though without the stronger exponential convergence statement of Proposition \ref{exponentialdecay}). 

In addition, for the  $\text{ADHM}_{1,2}$ Seiberg-Witten equations, the arguments of Section 3 may be extended to partially deal with the term  $\br q_1, f(q_0, q_0)\kt$  in the following way. Using Young's inequality and the fact that the Green's function (\refeq{RiemannianGreens}) is $L^2$-integrable in dimension 3, Proposition \ref{perturbationcase} can be adapted to prove that diverging spinors have the form $$\Phi_i=\Phi_0 + \ph_i $$
\noindent where $\ph_i=O(\e)$ in $C^0_{loc}$. This provides a step in confirming the asymptotic expansions postulated in \cite{DWAssociatives} (Section 5.3).  

\end{rem}

\appendix
\section{An Extension in $n=3$ Dimensions}
\label{appendixA}
This appendix provides a minor strengthening of Corollary \ref{cornonlinear} and Proposition \ref{exponentialdecay} to include the case of a family of compact sets $K_\e$ parameterized by $\e$. These results are employed in Corollary 1.3 and Appendix A of \cite{PartI} and in \cite{PartIII}.

Let $K_\e\Subset \mathcal Z_0$ be an $\e$-parameterized family of compact subsets in the complement of $\mathcal Z_0$. For a small constant $c_1$ to be chosen momentarily, denote by $R_\e=c_1\cdot \text{dist}(K_\e, \mathcal Z)$. We restrict to the case that $R_\e\to 0$, else we are in the previous situation with $K=\bigcup_\e K_\e$. Let $K_\e'$ denote a family of slightly larger compact subsets such that for some $\e_0>0$, the following conditions are met for all $\e<\e_0$:  
  
  \newpage 
\begin{enumerate}
\item[(1)] For some constant $\kappa_1<1$, one has $$\text{dist}(K_\e, Y- K_\e')\geq \kappa_1R_\e.$$
\item[(2)] If $y_0\in K_\e$ and $y\in B_{R_\e}(y_0)$ then $$ |\Lambda(y)|^2 \geq \frac{|\Lambda(y_0)|^2}{2}.$$
\item[(3)] The bounds \be \frac{\| A_\e\|_{L^{1,3}(K'_\e)}}{\inf_{K_\e} |\Lambda(y)|} \to 0 \hspace{2cm} \|A_\e\|_{C^0(K'_\e)} \leq c_2 \inf_{K_\e} |\Lambda(y)|\label{7.1}\ee are satisfied for all $\e<\e_0$ where $c_2<1/8$. 
\end{enumerate}

When assumptions (1)--(3) above are satisfied,  and $c_1$ is chosen sufficiently small, the following extension of Theorem \ref{maina} and Corollary \ref{cornonlinear} holds. 

\begin{cor}
Let $D_\e$ be a concentrating Dirac operator with fixed degeneracy, and $Q$ a non-linear operator satisfying the hypotheses of Corollary \ref{cornonlinear}. Suppose that $K_\e\Subset K_\e'\Subset Y-\mathcal Z$ are nested family of compact subsets of $Y-\mathcal Z_0$. If the family satisfies assumptions (1)--(3) above. Then the conclusion of Theorem \ref{concentrationprinciple} continues to hold, i.e. there exist constants $C, c$ independent of $\e$ and an $\e_0>0$ such that for $\e<\e_0$, any solution $$(D+\tfrac{1}{\e}\mathcal A)\frak q+Q(\frak q)=0$$ satisfies $$\|\pi_{\frak H}q\|_{C^0(K_\e)}\leq \frac{C}{|\text{dist}(K_\e, \mathcal Z)|^{3/2}}\text{Exp}\left(-\frac{\Lambda_{K_\e}c}{\e}\text{dist}(K_\e,\mathcal Z)\right)\|\frak q\|_{L^{1,2}(K'_\e)}.$$
\label{expdecaycor}
\end{cor}

\begin{proof}
The proof is identical to Corollary \ref{cornonlinear}, except that now for each point $y_0$ the radius of the ball used in Lemma \ref{greensfunction} depends on $\e$. In addition, the second bound of (\refeq{7.1}) is used to obtain (\refeq{absorbC0}) in this case, and the first is used to absorb the terms of Lemma \ref{interpolationclaim}. 
\end{proof}

To conclude, we note a particular case of interest. Suppose that in Case (I), $(\mathcal Z_0, A_0,\Phi_0)$ is a $\Z_2$-harmonic spinor which is  {\bf non-degenerate} in the sense that there exists a $c_2>0$ such that $$|\Phi_0|(x) \geq c_2\sqrt{ \text{dist}(x,\mathcal Z_0)}.$$

\noindent Then one has $\Lambda_{K_\e} \geq c_2\sqrt{R_\e}=c_2\sqrt{c_1}\cdot \text{dist}(K_\e, \mathcal Z_0)$. In this case, the Corollary \ref{expdecaycor} implies: 

\begin{cor}\label{lengthcor}
Suppose $(\mathcal Z_0, A_0,\Phi_0)$ is a non-degenerate $\Z_2$-harmonic spinor, and $K_\e= \{ y \ | \ \text{dist}(y,\mathcal Z_0)\geq c_1 \e^{2/3}\}$ is the family of compact subsets above. Then if the deformation $(\ph_\e, a_\e)$ satisfies (\refeq{fiducialbounds}), then 

 \be \|(\ph^\text{Im}, a)\|_{C^0(K_\e)} \leq \frac{C}{|\text{dist}(K_\e, \mathcal Z)|^{3/2}} \text{Exp}\left(-\frac{c_1}{\e}\text{dist}(K_\e, \mathcal Z_0)^{3/2}\right) \|(\ph,a)\|_{L^{1,2}(K_\e')}\label{charlength}\ee
 \medskip 
 \qed
 \label{CorA2}
\end{cor}
Corollary \ref{CorA2} implies the conclusions asserted in  Corollary 1.3 and Appendix B of \cite{PartI}. For the situation described there, one has that $K_\e= \{ y \ | \ \text{dist}(y,\mathcal Z_0)\geq c_1 \e^{2/3-\gamma_1}\}$ for $\gamma<<1$, the hypotheses (1)--(2) are easily seen to be satisfied. Hypothesis (3) follows from the slightly stronger assertion that 

\be \| \nabla \ph_\e\|_{L^{3}(K_\e)} \leq \frac{C}{\e^{2/3}}  \hspace{1cm} \text{and} \hspace{1cm}\|\ph_\e\|_{C^0(K_\e)}\leq \frac{C}{ \e^{2/3}}\label{fiducialbounds}.\ee

\noindent The results of Appendix B of \cite{PartI} show that the bounds \refeq{fiducialbounds} are satisfied for the model solutions used in the gluing construction of \cite{PartI,PartIII}. Thus hypotheses (1)--(3) hold, and Corollary \ref{lengthcor} implies the assertions in Corollary 1.3 and Appendix B of \cite{PartI}.

Corollary \ref{lengthcor} also implies a characteristic length scale that is not obvious from the convergence statements of Theorem \ref{compactness}. In particular, the exponential decay result applies on families of compact subsets with $\text{dist}(K_\e,\mathcal Z_0)\sim \e^{2/3}$. For smaller compact subsets, the exponent factor approaches $1$ and the conclusion of (\refeq{charlength}) becomes trivial. This suggests that $r=O(\e^{2/3})$ is a characteristic length scale for the convergence to a $\Z_2$-harmonic spinor in the case of the two-spinor Seiberg-Witten equations. Indeed, this length scale naturally appears in the construction of the model solutions used in the gluing construction of \cite{PartI,PartIII}. It is also the same length scale the appears in the equivalent problem for Hitchin's equations  \cite{MWWW, FredricksonSLnC, HitchinAsymptoticGeometry}. A promising approach to the  surjectivity of the gluing problem is to attempt to dilate this characteristic length to be unit size to extract limiting profiles of the sequence $(\ph_\e, a_\e)$ along the singular set $\mathcal Z_0$ and show these necessarily arise from gluing data as in \cite{PartI,PartIII}.

\section{Estimates for the Green's Function}
\label{Greensappendix}

This appendix proves assertion (iii) in the proof of Lemma \ref{greensfunction}. This is a consequence of the following: 

\begin{lm} Let $(X,g)$ be a Riemannian manifold of dimension 3 or 4 with bounded geometry, and take $R_0$ less than the injectivity radius of $X$. Let $M>0$, and for a point $x_0\in X$, denote by $G(x_0, x)$ the Green's function with Dirichlet boundary condition on $B_{R_0}(x_0)$, so that 

$$\begin{cases} \left(-\Delta_g - \frac{M^2}{\e^2}\right)G(x_0, x)= \delta_{x_0} \ \hspace{1cm}x\in B_{R_0}\\
G(x_0, x) =0 \hspace{3.5cm} x\in \del B_{R_0}
\end{cases}$$
\noindent where $\Delta_g$ denotes the (positive-definite) Laplacian of the Riemannian metric $g$. Then, there exists an $\e_0>0$ depending only on the geometry of $(X,g)$ and a constant $C_n$ depending only on the dimension $n$ such that the bound 

\be |G(x_0,x)| \leq \frac{C_n}{|x-x_0|^{n-2}} \ \text{Exp}\left(-\frac{M}{2\e}|x-x_0|\right)\label{greenestappendix}\ee
\noindent holds for $\e<\e_0$. 
\end{lm}

\begin{proof}This follows from a comparison principle argument using the Green's function on Euclidean space. We prove the lemma in dimension $n=3$; the general case is the same using the appropriate power of $|x-x_0|$. 

Let $\Delta_0$ denote the Laplacian on $\R^3$ with the Euclidean metric $g_0$, and $G_0^{M/2}$ denote the Green's function $$\left(-\Delta_g - \frac{M^2}{4\e^2}\right)G_0^{M/2}=\delta_0$$
on all of $\R^3$. It is a standard fact that $$G_0^{M/2}= \frac{C_3}{|x-x_0|} \text{Exp}\left(-\frac{M}{2\e}|x-x_0|\right).$$

Using geodesic normal coordinates on $B_{R_0}(x_0)$, define $\ph= G-G_0^{M/2}$ where $G=G(x_0,x)$. Then 
\bea
(-\Delta_g -\tfrac{M^2}{\e^2})\ph &=& \delta_{x_0}+ (\Delta_0 -\Delta_g)G_0^{M/2} -(\Delta_0 + \tfrac{M^2}{4\e^2})G_0^{M/2} - \tfrac{3M^2}{4\e^2} G_0^{M/2}\\
&=& 0 +(\Delta_0 - \Delta_g)G_0^{M/2} -\tfrac{3M^2}{4\e^2}G_0^{M/2}.
\eea
\noindent Since $g=g_0 + O(r^2)$ in geodesic normal coordinates, for a radially symmetric function $f(r)$ one has 
$$(\Delta_0 - \Delta_g)f(r)=O(r^2)\left[\del_r^2 + \frac{1}{r}\del_r\right] f(r)+O(r)\del_r f(r)=O(r^2)\del_r^2 + O(r)\del_r$$
\noindent where $O(r^k)$ denotes a quantity bounded by $Cr^k$ for a constant $C$ depending only on the geometry of $(X,g)$. Differentiating,  
\bea
\del_r G_0^{M/2}&=& C \left(-\frac{1}{r^2}- \frac{M}{2\e r}\right) e^{-\frac{M}{2\e}r}\\
\del_r^2 G_0^{M/2}&=& C \left(\frac{2}{r^3}+ \frac{M}{2\e}+ \frac{M^2}{4\e^2 r}\right) e^{-\frac{M}{2\e}r}
\eea
\noindent hence $$|(\Delta_0-\Delta_g) G_0^{M/2}|\leq C\left(\frac{1}{r^3}+ \frac{M}{2\e}+ \frac{M^2}{4\e^2 r}\right) e^{-\frac{M}{2\e}r}\leq \frac{3M^2}{4\e^2} G_0^{M/2}
 $$
 \noindent once $\e$ is sufficiently small depending only on $R_0$ and the geometry of $(X,g)$. It follows that $\ph$ satisfies 
 \bea
 (\Delta_g + \tfrac{M^2}{\e^2})\ph &\leq& 0 \hspace{1cm}\text{on}\hspace{1cm}B_{R_0}(x_0) \\ 
 \ph &\leq& 0 \hspace{1cm}\text{on}\hspace{1cm}\del B_{R_0}(x_0)
 \eea
 where the last line follows since $G=0$ on $\del B_{R_0}$ by definition, and $-G_0^{M/2}<0$ everywhere. The maximum principle implies that $\ph \leq 0$ on $B_{R_0}$ once $\e$ is sufficiently small, which yields the desired estimate (\refeq{greenestappendix}).
\end{proof}


{\small

\medskip

\bibliographystyle{amsplain}
{\small 
\bibliography{Bibliography_Collected_ConcentratingDirac}}
\end{document}



\end{document}